\documentclass[draft,reqno]{amsproc}
\usepackage{amsthm}
\usepackage{amsfonts}
\usepackage{amsmath}
\usepackage{amssymb}
\usepackage{mathrsfs}
\usepackage{euscript}   
\usepackage{color}
\usepackage{mathbbol}
\usepackage{pifont}
\usepackage{tikz}

\makeatletter
\@namedef{subjclassname@2010}{%
\textup{2010} Mathematics Subject Classification}
\makeatother

\newtheorem{thm}{Theorem}[section]
\newtheorem{pro}[thm]{Proposition}
\newtheorem{cor}[thm]{Corollary}
\newtheorem{lem}[thm]{Lemma}

\newtheorem*{opq*}{\bf Problem}

\theoremstyle{remark}
\newtheorem{rem}[thm]{Remark}

\theoremstyle{definition}
\newtheorem{dfn}[thm]{Definition}
\newtheorem{exa}[thm]{Example}

\newcommand*{\borel}[1]{{\mathfrak B}(#1)}
\newcommand*{\bscr}{\mathscr{B}}
\newcommand*{\cbb}{\mathbb{C}}
\newcommand*{\dbb}{\mathbb{D}}
\newcommand*{\dbbc}{\bar{\mathbb{D}}}
\newcommand*{\gqb}{\mathcal{Q}}
\newcommand*{\gqbh}{\mathcal{Q}_{\hh_1,\hh_2}}
\newcommand*{\gqbhh}[1]{\mathcal{Q}_{{#1}}}
\newcommand*{\Ge}{\geqslant}
\newcommand*{\hh}{\mathcal{H}}

\newcommand*{\is}[2]{\langle#1,#2\rangle}
\newcommand*{\jd}[1]{\mathscr N(#1)}
\newcommand*{\kk}{\EuScript{K}}
\newcommand*{\Le}{\leqslant}
\newcommand*{\mcal}{\mathscr{M}}
\newcommand*{\nbb}{\mathbb N}
\newcommand*{\ob}[1]{{\mathscr R}(#1)}
\newcommand*{\ogr}[1]{\boldsymbol B(#1)}
\newcommand*{\rbb}{\mathbb R}

\newcommand*{\subE}[1]{{\mathscr S}^{\dag}(#1)}
\newcommand*{\subQ}[1]{{\mathscr S}_{\dag}(#1)}

\newcommand*{\supp}{\mathrm{supp}\,}
\newcommand*{\tbb}{\mathbb{T}}
\newcommand*{\zbb}{\mathbb Z}
   
\begin{document}
   \title[Taylor spectrum approach to Brownian-type
operators] {Taylor spectrum approach to Brownian-type \\
operators with quasinormal entry}
   \author[S. Chavan]{Sameer Chavan}
   \address{Department of Mathematics and Statistics\\
Indian Institute of Technology Kanpur, India}
   \email{chavan@iitk.ac.in}
   \author[Z.\ J.\ Jab{\l}o\'nski]{Zenon Jan
Jab{\l}o\'nski}
   \address{Instytut Matematyki,
Uniwersytet Jagiello\'nski, ul.\ \L ojasiewicza 6,
PL-30348 Kra\-k\'ow, Poland}
\email{Zenon.Jablonski@im.uj.edu.pl}
   \author[I.\ B.\ Jung]{Il Bong Jung}
   \address{Department of Mathematics,
Kyungpook National University, Daegu 702-701, Korea}
   \email{ibjung@knu.ac.kr}
   \author[J.\ Stochel]{Jan Stochel}
\address{Instytut Matematyki, Uniwersytet
Jagiello\'nski, ul.\ \L ojasiewicza 6, PL-30348
Kra\-k\'ow, Poland} \email{Jan.Stochel@im.uj.edu.pl}
   \thanks{The research of the third author was supported by the National Research
Foundation of Korea (NRF) grant funded by the Korea
government (MSIT) (2018R1A2B6003660)}
   \begin{abstract}
In this paper, we introduce operators that are
represented by upper triangular $2\times 2$ block
matrices whose entries satisfy some algebraic
constraints. We call them Brownian-type operators of
class $\gqb,$ briefly operators of class $\gqb.$ These
operators emerged from the study of Brownian
isometries performed by Agler and Stankus via detailed
analysis of the time shift operator of the modified
Brownian motion process. It turns out that the class
$\gqb$ is closely related to the Cauchy dual
subnormality problem which asks whether the Cauchy
dual of a completely hyperexpansive operator is
subnormal. Since the class $\gqb$ is closed under the
operation of taking the Cauchy dual, the problem
itself becomes a part of a more general question of
investigating subnormality in this class. This issue,
along with the analysis of nonstandard moment
problems, covers a large part of the paper. Using the
Taylor spectrum technique culminates in a full
characterization of subnormal operators of class
$\gqb.$ As a consequence, we solve the Cauchy dual
subnormality problem for expansive operators of class
$\gqb$ in the affirmative, showing that the original
problem can surprisingly be extended to a class of
operators that are far from being completely
hyperexpansive. The Taylor spectrum approach turns out
to be fruitful enough to allow us to characterize
other classes of operators including $m$-isometries.
We also study linear operator pencils associated with
operators of class $\gqb$ proving that the
corresponding regions of subnormality are closed
intervals with explicitly described endpoints.
   \end{abstract}
   \maketitle
   \newpage
   \tableofcontents
   \section{Introduction}
Given two complex Hilbert spaces $\hh$ and $\kk,$ we
denote by $\ogr{\hh,\kk}$ the Banach space of all
bounded linear operators from $\hh$ to $\kk$. The
kernel, the range, the adjoint and the modulus of an
operator $T \in \ogr{\hh,\kk}$ are denoted by $\jd T,$
$\ob T,$ $T^*$ and $|T|,$ respectively. We regard
$\ogr{\hh}:=\ogr{\hh,\hh}$ as a $C^*$-algebra. The
identity operator on $\hh$ is denoted by $I_\hh,$ or
simply by $I$ if no ambiguity arises. Recall that an
operator $T\in \ogr{\hh}$ is said to be {\em
quasinormal} if $TT^*T=T^*TT,$ or equivalently if
$T|T|=|T|T.$ We say that $T$ is {\em subnormal} if
there exist a complex Hilbert space $\kk$ and a normal
operator $N\in \ogr{\kk}$ such that $\hh \subseteq
\kk$ (an isometric embedding) and $Sh=Nh$ for all
$h\in \hh.$ It is well known that quasinormal
operators are subnormal (see
\cite[Proposition~II.1.7]{Co}). We refer the reader to
\cite{Co} for more information on these classes of
operators.

Let $T\in \ogr{\hh}$. We say that $T$ is a {\em
$2$-isometry} if $T^{*2}T^{2} - 2 T^*T + I =0.$ We call $T$
a {\em Brownian isometry} if $T$ is a $2$-isometry such
that $\triangle_T \triangle_{T^*} \triangle_T=0,$ where
$\triangle_T =T^*T-I.$ If $\triangle_T \Ge 0$ and
$\triangle_T T = \triangle_T^{1/2} T \triangle_T^{1/2},$ we
say that $T$ is {\em $\triangle_T$-regular}. By a {\em
quasi-Brownian isometry} we mean a $\triangle_T$-regular
$2$-isometry. It is well known that any $2$-isometry is
left-invertible\footnote{\;In this paper,
left-invertibility and invertibility of an operator $T\in
\ogr{\hh}$ refer to the algebra $\ogr{\hh}.$} and
$\triangle_T \Ge 0$ (\cite[Lemma~1]{R-0}). The notion of a
$2$-isometry was invented by Agler in \cite{Ag-0}, while
the notion of a Brownian isometry was introduced by Agler
and Stankus in \cite{Ag-St}. The class of $2$-isometric
operators emerged from the study of the time shift operator
of the modified Brownian motion process from one side
\cite{Ag-St}, and from the investigation of invariant
subspaces of the Dirichlet shift from the other \cite{R-0}.
The class of $\triangle_T$-regular $2$-isometries were
investigated in \cite{Maj,B-S} and in
\cite{A-C-J-S,A-C-J-S-2} under the name of quasi-Brownian
isometries.

Given a left-invertible operator $T\in \ogr{\hh},$ we set
$T^{\prime}=T(T^*T)^{-1}.$ Following \cite{Sh}, we call
$T^{\prime}$ the {\em Cauchy dual operator} of $T.$ Recall
that if $T$ is left-invertible, then so is $T^{\prime}$ and
$T=(T')'.$ Athavale noticed that the Cauchy dual operator
of a completely hyperexpansive injective unilateral
weighted shift is a subnormal contraction (see
\cite[Proposition~ 6]{Ath} with $t=1$), but not conversely
(see \cite[Remark~ 4]{Ath}). The {\em Cauchy dual
subnormality problem} asks whether the Cauchy dual operator
of a completely hyperexpansive operator (see
Section~\ref{Sec8} for the definition) is a subnormal
contraction (see \cite[Question 2.11]{Ch-0}). As shown in
\cite{A-C-J-S}, the answer is in the negative even for
$2$-isometries, that is, there are $2$-isometries whose
Cauchy dual operators are not subnormal (recall that each
$2$-isometry is completely hyperexpansive and that the
Cauchy dual operator of a completely hyperexpansive
operator is always a contraction). However, as proved in
\cite[Theorem~ 4.5]{A-C-J-S}, the Cauchy dual operator
$T^{\prime}$ of a quasi-Brownian isometry $T$ is a
subnormal contraction (see also \cite[Theorem~ 3.4]{B-S}
for a recent generalization of this result to the case of
completely hyperexpansive $\triangle_T$-regular operators).
This leads to the question of why this phenomenon can
happen. We will try to answer it by regarding
quasi-Brownian isometries as elements of a larger class of
operators which is closed under the operation of taking the
Cauchy dual (note that the class of quasi-Brownian
isometries is not closed under this operation). As a
consequence, in the larger class of operators, the Cauchy
dual subnormality problem becomes a part of the more
general question of finding necessary and sufficient
conditions for subnormality.

Let us recall that non-isometric Brownian and
quasi-Brownian isometries have upper triangular $2\times 2$
block matrix representations with entries satisfying some
algebraic constraints (see the remark just after
Definition~\ref{defq}). For the purposes of our paper
explained in the above discussion, we introduce a wider
class of operators consisting of the so-called Brown-type
operators.
   \begin{dfn} \label{defq} We say that
$T\in \ogr{\hh}$ is a {\em Brownian-type operator} if it
has the block matrix form
   \begin{align} \label{brep}
T = \begin{bmatrix} {V } & {E}\\ {0} & {Q}
   \end{bmatrix}
   \end{align}
with respect to a nontrivial\footnote{\;Nontriviality
means that $\hh_1 \neq \{0\}$ and $\hh_2 \neq \{0\}.$}
orthogonal decomposition $\hh=\hh_1 \oplus \hh_2$,
where the operators $V\in \ogr{\hh_1},$ $E\in
\ogr{\hh_2,\hh_1}$ and $Q\in \ogr{\hh_2}$ satisfy the
following conditions:
   \begin{gather}  \label{gqb-1}
\text{$V$ is an isometry, i.e., $V^*V=I,$}
   \\  \label{gqb-2}
V^*E=0,
   \\  \label{gqb-3}
QE^*E=E^*EQ.
   \end{gather}
Moreover, if
   \begin{align}  \label{gqb-4}
\text{$Q$ is quasinormal,}
   \end{align}
we call $T$ a {\em Brownian-type operator of class
$\gqb$} and write
$T =  \big[\begin{smallmatrix} V & E \\
0 & Q \end{smallmatrix}\big] \in \gqbh;$ to simplify
the terminology, we say that $T$ is an {\em operator
of class $\gqb$}. By analogy, if $Q$ is isometric
(resp.\ unitary, normal, etc.), then $T$ is called an
operator of {\em class} $\mathcal{I}$ (resp.\
$\mathcal{U}$, $\mathcal{N},$ etc.). If $\kk$ is a
complex Hilbert space and $\hh=\kk \oplus \kk$
(understood as an external orthogonal sum), then we
abbreviate $\mathcal{Q}_{\kk,\kk}$ to $\gqbhh{\kk}.$
   \end{dfn}
In Definition \ref{defq}, we have decided to exclude
the case when one of the summands $\hh_1$ or $\hh_2$
is absent because otherwise the operator $T$ is
quasinormal. Moreover, by \eqref{gqb-1} and
\eqref{gqb-2}, the hypothesis that $E\neq 0$ excludes
the case when $\hh_1$ is finite dimensional. Notice
also that by the square root theorem
\cite[Theorem~2.4.4]{Sim-4}, the equality
\eqref{gqb-3} is equivalent to $Q|E|=|E|Q.$ One can
deduce from \cite[Proposition~5.37 and
Theorem~5.48]{Ag-St} (resp., \cite[Proposition~
5.1]{Maj}) that a non-isometric operator $T\in
\ogr{\hh}$ is a Brownian isometry (resp., a
quasi-Brownian isometry) if and only if $T$ is of
class $\mathcal{U}$ (resp., of class $\mathcal{I}$)
(to avoid injectivity of $E$ postulated in
\cite[Proposition~5.37]{Ag-St} and
\cite[Proposition~5.1]{Maj}, consult
\cite[Theorem~4.1]{A-C-J-S}). This means that Brownian
isometries are quasi-Brownian isometries. In view of
\cite[Example~ 4.4]{A-C-J-S}, the converse implication
is not true in general.

It is worth pointing out that upper triangular
$2\times 2$ block matrices appear in different parts
of operator theory and functional analysis on the
occasion of investigating variety of topics; for
example, the hyperinvariant subspace problem
\cite{Do-Pe72,Ki12,JKP18,JKP19}, the Halmos similarity
problem for polynomially bounded operators
\cite{Fog64,Pi97}, the task of finding models for the
time shift operator for modified Brownian motion
process \cite{Ag-St}, the question of characterizing
invertibility of upper triangular $2\times 2$ block
matrices \cite{WLee00}, the task of searching for a
model theory for $2$-hyponormal operators
\cite{Cur-Lee02}, the problem of determining a
complete set of unitary invariants for the class of
Cowen-Douglas operators realized as upper triangular
$2\times 2$ block matrices \cite{Mis17}, and many
others.

We state now the main result of this paper which
characterizes subnormality of operators of class
$\gqb$ in terms of the Taylor spectrum
$\sigma(|Q|,|E|)$ of the pair $(|Q|,|E|).$ The
spectral region for subnormality of operators of class
$\gqb$ is described by Theorem~\ref{main}(iii) and
illustrated in Figure~\ref{fig6}. We refer the reader
to Section~\ref{Sec2.5} for the necessary definitions
and notations.
   \begin{thm} \label{main}
Suppose $T = \big[\begin{smallmatrix} V & E \\
0 & Q \end{smallmatrix}\big] \in \gqbh.$ Let $P\in
\ogr{\hh_2}$ be the orthogonal projection of $\hh_2$
onto $\mcal:=\overline{\ob{|E|}}.$ Then the operators
$|Q|,$ $|E|$ and $P$ commute, $\mcal$ reduces $|Q|$
and $|E|,$ and the following conditions are
equivalent{\em :}
   \begin{enumerate}
   \item[(i)] $T$ is subnormal,
   \item[(ii)] $\sigma_{\sharp}(|Q|,|E|) \subseteq \dbbc_+,$
where $\sigma_{\sharp}(|Q|,|E|):=\sigma(|Q|,|E|) \cap
(\rbb_+ \times (0,\infty)),$
   \item[(iii)] $\sigma(|Q|,|E|) \subseteq \dbbc_+ \cup (\rbb_+ \times
\{0\}),$
   \item[(iv)] $(|Q|P,|E|)$  is a spherical contraction,
   \item[(v)] $(|Q|\big|_{\mcal},|E|\big|_{\mcal})$
is a spherical contraction,
   \item[(vi)] $\sigma(|Q|\big|_{\mcal},|E|\big|_{\mcal})
\subseteq \dbbc_+.$
   \end{enumerate}
Moreover, if $T$ is subnormal, then
   \begin{align*}
\sigma(|Q|,|E|) \subseteq \big(\dbbc_+ \cup (\rbb_+ \times
\{0\})\big) \cap \big(\sigma(|Q|) \times \sigma(|E|)\big).
   \end{align*}
   \end{thm}
   \begin{figure}
   \begin{tikzpicture}[scale=.35, transform shape]
\tikzset{vertex/.style = {shape=circle,draw,minimum
size=1em}} \tikzset{edge/.style = {->,> = latex'}}

\node[] (1) at (-6.5, -6.5) {$\scalebox{1.5}{(0, 0)}$};

\node[] (1) at (0.0, -6.5) {$\scalebox{1.5}{(1, 0)}$};

\node[] (1) at (-6.9, 0.0) {$\scalebox{1.5}{(0, 1)}$};

\draw[->, line width=0.5mm] (-6,-6) -> (4,-6);

\draw[line width=0.5mm] (-6,-6) -> (-6,0);

\draw[->, line width=0.05mm] (-6,0) -> (-6,3);

\def\Radius{6}

\filldraw[fill opacity=0.7, line width=0.5mm,
fill=gray] (-6,-6) -- (0,-6) arc (0:90:6cm) -- cycle;
   \end{tikzpicture}
   \caption{Spectral region for subnormality of operators
of class $\gqb.$} \label{fig6}
   \end{figure}
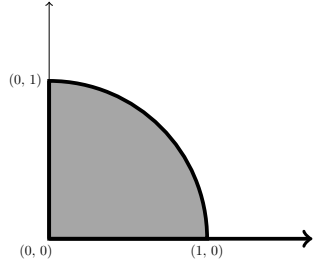
The proof of Theorem~\ref{main} is fairly long and it
occupies most of Sections~\ref{Sec2}, \ref{Sec3}
and~\ref{Sec4}. The theorem itself has many applications
spread over Sections~ \ref{Sec4}, \ref{Sec6} and
\ref{Sec9}. In particular, we show that contractions of
class $\gqb$ are subnormal (see Corollary~\ref{main-cor}),
we solve the Cauchy dual subnormality problem for expansive
operators of class $\gqb$ in the affirmative (see
Corollary~\ref{cd-im-su}) and, what is more important, we
completely characterize subnormality of the Cauchy-duals of
left-invertible operators of class~$\gqb$ (see
Theorem~\ref{main-Ca}). The study of linear operator
pencils associated with operators of class $\gqb$ provides
a useful test of the applicability of the main theorem (see
Theorems~\ref{subex} and \ref{pen-2}).

The Taylor spectrum approach developed in this paper
for the purpose of investigating subnormality turns
out to be efficient when studying other collections of
operators of class $\gqb$ including $m$-contractions,
$m$-isometries, etc. (see Section~\ref{Sec8}). In
fact, it appears to be effective even in providing
explicit formulas for the norm of operators of class
$\gqb$ (see \eqref{norm-u} in Section~\ref{Sec2}) and
for the right endpoints of the intervals of
subnormality of linear operator pencils associated
with operators of class $\gqb$ (see \eqref{Zen-claps}
and \eqref{Zen-claps-2} in Section~\ref{Sec9}). The
Taylor spectrum technique is also applied to
characterize quasi-Brownian and Brownian isometries of
class $\gqb$ in Sections~\ref{Sec7} and \ref{Sec7.5},
respectively. Unexpectedly, the Brownian case is
essentially more complicated. The reader has to be
aware of the fact that quasi-Brownian (and so
Brownian) isometries are always of class $\gqb,$
however relative to properly selected orthogonal
decompositions (of the underlying Hilbert spaces),
which are not necessarily easy to be determined in
concrete cases.

We conclude Introduction by pointing out that the
overwhelming majority of the characterizations of
selected subclasses of the class $\gqb$ that appear in
this paper consist in finding for a given subclass a
minimal universal subset of the Euclidean plane having
the property that an operator $T = \big[\begin{smallmatrix} V & E \\
0 & Q \end{smallmatrix}\big] \in \gqbh$ belongs to the
subclass if and only if the Taylor spectrum
$\sigma(|Q|,|E|)$ of the pair $(|Q|,|E|)$ is contained
in the aforementioned subset. The universality of this
subset lies in the fact that it does not depend on the
choice of the orthogonal decomposition $\hh_1 \oplus
\hh_2$ of the underlying Hilbert space $\hh$ relative
to which a given operator $T\in \ogr{\hh}$ is of class
$\gqb,$ i.e., $T$ has the block matrix form
\eqref{brep} with $V,$ $E$ and $Q$ satisfying
\eqref{gqb-1}-\eqref{gqb-4}. What is more interesting,
there may exist different orthogonal decompositions of
$\hh$ relative to which the given operator $T$ is of
class $\gqb$ and the Taylor spectra $\sigma(|Q|,|E|)$
of the corresponding pairs $(|Q|,|E|)$ are
significantly different (see Example~\ref{wid-dmo-1}).
It turns out that the class of Brownian isometries is
the only subclass of $\gqb$ considered in this paper
which cannot be characterized by the Taylor spectrum
$\sigma(|Q|, |E|)$ of the pair $(|Q|,|E|)$ (see
Remark~\ref{haha-2}).
   \section{\label{Sec2.5}Prerequisites}
   In this section we fix notation and terminology and give
necessary facts. Let $\zbb$, $\rbb$ and $\cbb$ stand for
the sets of integers, real numbers and complex numbers,
respectively. Denote by $\nbb$ the set of positive
integers. Set
   \begin{gather*}
\zbb_+=\{n\in \zbb\colon n \Ge 0\}, \quad
\rbb_+=\{x\in \rbb\colon x\Ge 0\},
   \\
\dbb_+=\{(s, t) \in \rbb_+^2 \colon s^2 + t^2 < 1\},
\quad \tbb_+=\{(s,t)\in \rbb_+^2\colon s^2 + t^2=1\},
   \\
\dbbc_+ = \dbb_+ \cup \tbb_+.
   \end{gather*}
Given a set $X$, we write $\chi_{\varDelta}$ for the
characteristic function of a subset $\varDelta$ of
$X$. The $\sigma$-algebra of all Borel subsets of a
topological space $X$ is denoted by $\borel{X}$. For
$x\in X,$ $\delta_{x}$ stands for the Borel
probability measure on $\rbb$ supported on $\{x\}.$

Let $\hh$ be a complex Hilbert space. We call an operator
$T\in \ogr{\hh}$ a {\em contraction} (resp., an {\em
expansion}) if $\|Th\| \Le \|h\|$ for all $h\in \hh$
(resp., $\|Th\| \Ge \|h\|$ for all $h\in \hh$), or
equivalently if $T^*T \Le I$ (resp., $T^*T \Ge I$). The
contractivity of $T$ can also be characterized by requiring
that $\|T\|\Le 1$ (however $\|T\|\Ge 1$ does not
characterize expansivity of $T$). Obviously, $T$ is an
isometry if and only $T$ is simultaneously a contraction
and an expansion. We write $\sigma(T)$ for the spectrum of
$T$. If $G$ is a regular Borel spectral measure on a
topological Hausdorff space $X$, then $\supp{G}$ denotes
the {\em closed support} of $G,$ i.e., $X\setminus
\supp{G}$ is the largest open subset $\varDelta$ of $X$
such that $G(\varDelta)=0.$ Recall that if $T\in \ogr{\hh}$
is a selfadjoint operator and $G$ is the spectral measure
of $T,$ then $\sigma(T) = \supp{G}.$ The following
elementary fact will be frequently used in this paper.
   \begin{align} \label{font-1}
   \begin{minipage}{70ex}
{\em Suppose that $T\in \ogr{\hh}$ is selfadjoint. If $a,b
\in \rbb$ are such that $a\Le b,$ then $\sigma(T) \subseteq
[a,b]$ if and only if $a I \Le T \Le b I.$ Moreover, if $T
\Ge 0$ and $0\notin \sigma(T),$ then $\min \sigma(T)=
\|T^{-1}\|^{-1}$ and $\max \sigma(T)= \|T\|.$}
   \end{minipage}
   \end{align}
We refer the reader to \cite[Chapter~6]{Bi-So} for more
details on spectral theory of Hilbert space operators.

A pair $(T_1,T_2)$ of commuting operators $T_1,T_2 \in
\ogr{\hh}$ is said to be a {\em spherical contraction}
(resp., {\em spherical expansion}) if $T_1^*T_1 + T_2^*T_2
\Le I$ (resp., $T_1^*T_1 + T_2^*T_2\Ge I$). If $(T_1,T_2)$
is simultaneously spherical contraction and spherical
expansion, that is $T_1^*T_1 + T_2^*T_2 = I$, then
$(T_1,T_2)$ is called a {\em spherical isometry} (see
\cite{Ath90}).

For a pair $(T_1,T_2)$ of commuting operators $T_1,T_2 \in
\ogr{\hh}$, we denote by $\sigma(T_1,T_2)$ the {\em Taylor
spectrum} of $(T_1,T_2),$ and by $r(T_1,T_2)$ the {\em
geometric spectral radius} of $(T_1,T_2),$ that is,
   \begin{align*}
r(T_1,T_2) = \max\Big\{(|z_1|^2+|z_2|^2)^{1/2}\colon
(z_1,z_2) \in \sigma(T_1,T_2)\Big\}.
   \end{align*}
The reader is referred to
\cite{Tay70,Vas82,Cur88,Mu07,Ch-Ze92} for the definitions
and the basic properties of the Taylor spectrum and the
geometric spectral radius (of commuting $n$-tuples of
operators). In particular, the Taylor spectrum
$\sigma(T_1,T_2)$ is a nonempty compact subset of $\cbb^2$
whenever $\hh\neq \{0\}.$ Moreover, it has the following
{\em projection property} (see \cite[Lemma~3.1]{Tay70}; see
also \cite[Theorem~4.9]{Cur88}):
   \begin{align} \label{pro-pr-y}
\pi_j(\sigma(T_1,T_2)) = \sigma(T_j), \quad j=1,2,
   \end{align}
where $\pi_1,\pi_2\colon \cbb^2\to \cbb$ are defined by
$\pi_1(z_1,z_2)=z_1$ and $\pi_2(z_1,z_2)=z_2$ for
$(z_1,z_2)\in \cbb^2.$ The following fact follows directly
from the projection property of the Taylor spectrum.
   \begin{align} \label{tplus0}
   \begin{minipage}{74ex}
{\it Suppose that $\hh\neq \{0\}$ and $\lambda\in \cbb.$
Then $\sigma(T_1,T_2) \subseteq \{\lambda\} \times \cbb$ if
and only if $\sigma(T_1)=\{\lambda\}.$ Moreover, if
$\sigma(T_1)=\{\lambda\},$ then $\sigma(T_1,T_2) =
\{\lambda\} \times \sigma(T_2).$ The symmetric version with
$\cbb \times \{\lambda\}$ in place of $\{\lambda\} \times
\cbb$ holds as well.}
   \end{minipage}
   \end{align}
Note that under the assumptions of \eqref{tplus0},
$\sigma(T_1,T_2) = \sigma(T_1) \times \sigma(T_2)$ if
$\sigma(T_1)=\{\lambda\}$ or if $\sigma(T_2)=\{\lambda\}$.
However, the first equation may not hold even for positive
operators (see \eqref{sryp-u} in
Example~\ref{not-contr-1}).

For a given pair $(T_1,T_2)$ of commuting selfadjoint
operators $T_1,T_2 \in \ogr{\hh}$, there exists a
unique Borel spectral measure $G\colon \borel{\rbb^2}
\to \ogr{\hh}$, called the {\em joint spectral
measure} of $(T_1,T_2),$ such that
   \begin{align} \label{fr-fr-1}
p(T_1,T_2)=\int_{\rbb^2} p(t_1,t_2) G(d t_1, d t_2), \quad
p\in \cbb[x_1,x_2],
   \end{align}
where as usual $\cbb[x_1,x_2]$ stands for the ring of
polynomials in indeterminates $x_1,x_2$ with complex
coefficients (similar notations are used throughout the
paper with no further explanation). The joint spectral
measure $G$ is the product of the spectral measures of
$T_1$ and $T_2$ (see \cite[Theorem~6.5.1]{Bi-So}). As shown
below, in this particular case, the Taylor spectrum
$\sigma(T_1,T_2)$ coincides with the closed support of the
joint spectral measure $G$; this yields the spectral
mapping theorem for continuous functions.\footnote{\;Note
that Theorem~\ref{spmt} remains true for commuting normal
operators with $\cbb$ in place of $\rbb.$ We refer the
reader to \cite[Theorem~4.8]{Tay70-c} (see also
\cite[Theorem~5.19]{Cur88} and
\cite[Corollary~IV.30.11]{Mu07}) for the spectral mapping
theorem for the Taylor functional calculus.}
   \begin{thm} \label{spmt}
Suppose that $T_1,T_2 \in \ogr{\hh}$ are commuting
selfadjoint operators with the joint spectral measure
$G$. Then the following assertions are valid{\em :}
   \begin{enumerate}
   \item[(i)] $\sigma(T_1,T_2) = \supp{G};$
moreover, if $T_1, T_2$ are positive, then
$\sigma(T_1,T_2) \subseteq \rbb_+^2,$
   \item[(ii)] for any continuous function $\psi\colon
\sigma(T_1,T_2) \to \rbb,$
   \begin{align*}
\sigma(\psi(T_1,T_2)) = \psi(\sigma(T_1,T_2)),
   \end{align*}
where $\psi(T_1,T_2) := \int_{\sigma(T_1,T_2)} \psi \,
d G$,
   \item[(iii)] for any continuous function
$\boldsymbol{\psi}=(\psi_1,\psi_2)\colon \sigma(T_1,T_2)
\to \rbb^2,$
   \begin{align*}
\sigma(\boldsymbol{\psi}(T_1,T_2)) =
\boldsymbol{\psi}(\sigma(T_1,T_2)),
   \end{align*}
where $\boldsymbol{\psi}(T_1,T_2):=
(\psi_1(T_1,T_2),\psi_2(T_1,T_2)).$
   \end{enumerate}
   \end{thm}
   \begin{proof}
First observe that by \eqref{pro-pr-y} we have
   \begin{align} \label{Sam-Jan2Zen}
\sigma(T_1,T_2) \subseteq \sigma(T_1) \times
\sigma(T_2) \subseteq \rbb^2,
   \end{align}
so if additionally $T_1$ and $T_2$ are positive, then
$\sigma(T_1,T_2) \subseteq \rbb_+^2$.

(i) First note that the Taylor spectrum $\sigma(T_1,T_2)$
coincides with the left spectrum of $(T_1,T_2)$ (see
\cite[Proposition~7.2]{Cur88}). It is a routine matter to
show that the left spectrum of $(T_1,T_2)$ coincides with
the approximate point spectrum of $(T_1,T_2)$ (this is true
for an arbitrary pair of commuting Hilbert space
operators). Hence, for $(\lambda_1,\lambda_2) \in \rbb^2$,
$(\lambda_1,\lambda_2) \notin \sigma(T_1,T_2)$ if and only
if there exists $c\in (0,\infty)$ such that
   \begin{align*}
\|(T_1 - \lambda_1 I)h\| + \|(T_2 - \lambda_2 I)h\|
\Ge c\|h\|, \quad h \in \hh,
   \end{align*}
or equivalently, by \cite[Theorem~6.5.3]{Bi-So}, if and
only if $(\lambda_1,\lambda_2) \notin \supp{G}$. Combined
with \eqref{Sam-Jan2Zen}, this proves (i).

(ii) Note that
   \begin{align*}
\sigma(\psi(T_1,T_2)) =
\sigma\bigg(\int_{\sigma(T_1,T_2)} \psi \, d G\bigg)
\overset{(*)}= \psi(\supp{G}) \overset{\mathrm{(i)}}=
\psi(\sigma(T_1,T_2)),
   \end{align*}
where $(*)$ follows from \cite[eq.\ (13), p.\
158]{Bi-So}.

(iii) By \cite[Theorem~6.6.4]{Bi-So}, $G\circ
\psi_j^{-1}$ is the spectral measure of
$\psi_j(T_1,T_2)$ for $j=1,2.$ Let $\widetilde G$ be
the product of these measures (see
\cite[Theorem~5.2.6]{Bi-So}). Since
   \begin{align*}
\widetilde G (\varDelta_1 \times \varDelta_2) &= G
(\psi_1^{-1} (\varDelta_1)) G(\psi_2^{-1} (\varDelta_2)) =
G (\boldsymbol{\psi}^{-1}(\varDelta_1 \times \varDelta_2)),
\quad \varDelta_1, \varDelta_2 \in \borel{\rbb},
   \end{align*}
we deduce from the uniqueness part of
\cite[Theorem~5.2.6]{Bi-So} that $\widetilde G = G\circ
\boldsymbol{\psi}^{-1}.$ Hence $G\circ
\boldsymbol{\psi}^{-1}$ is the joint spectral measure of
the pair $\boldsymbol{\psi}(T_1,T_2).$ This yields
   \begin{align*}
\sigma(\boldsymbol{\psi}(T_1,T_2)) \overset{\mathrm{(i)}}=
\supp{G\circ \boldsymbol{\psi}^{-1}} \overset{(*)}=
\boldsymbol{\psi}(\supp{G}) \overset{\mathrm{(i)}}=
\boldsymbol{\psi}(\sigma(T_1,T_2)).
   \end{align*}
(To get $(*)$ adapt the proof of
\cite[Lemma~3.2]{St-St12}.) This completes the proof.
   \end{proof}
As a consequence of Theorem~\ref{spmt}, we obtain the
following.
   \begin{align} \label{stis-zero}
   \begin{minipage}{72ex}
{\em If $T_1,T_2 \in \ogr{\hh}$ are commuting and
selfadjoint operators, then $\sigma(T_1,T_2) \subseteq
\big(\rbb \times \{0\}\big) \cup \big(\{0\} \times
\rbb\big)$ if and only if $T_1T_2=0,$ or equivalently
if $T_1=0\oplus\widetilde{T}_1$ and $T_2=\widetilde
T_2\oplus 0$ relative to $\hh = \jd{T_1} \oplus
\overline{\ob{T_1}}.$}
   \end{minipage}
   \end{align}
For this, note that $\sigma(T_1,T_2) \subseteq \big(\rbb
\times \{0\}\big) \cup \big(\{0\} \times \rbb\big)$ if and
only if $p(\sigma(T_1,T_2))=0,$ where $p(s,t)=s\cdot t.$
Hence, applying Theorem~\ref{spmt}(ii) gives the former
equivalence in \eqref{stis-zero}. The latter is a matter of
routine verification.

The following lemma is surely folklore. For
self-containedness we sketch its proof (the reader can
easily formulate a version for commuting normal operators).
   \begin{lem} \label{inf-del}
Let $T_1,T_2 \in \ogr{\hh}$ be commuting selfadjoint
operators on a nonze\-ro complex Hilbert space $\hh.$ Then
   \begin{align} \label{impo-ss-0}
r(T_1, T_2) =\|T_1^2 + T_2^2\|^{1/2} = \min\big\{\delta\in
\rbb_+\colon \sigma(T_1, T_2) \subseteq \delta\cdot
\bar\dbb \big\},
   \end{align}
where $\dbb:=\big\{(s,t)\in \rbb^2 \colon s^2+t^2 <
1\big\}.$ Moreover, $T_1^2 + T_2^2$ is invertible if and
only if $\big\{\delta\in (0,\infty) \colon \sigma(T_1, T_2)
\subseteq \rbb^2 \setminus \delta\cdot \dbb\big\} \neq
\emptyset$; if this is the case, then
   \begin{align} \label{impo-ss}
\|(T_1^2 + T_2^2)^{-1}\|^{-1/2} = \max\big\{\delta\in
(0,\infty)\colon \sigma(T_1, T_2) \subseteq \rbb^2
\setminus \delta\cdot \dbb\big\}.
   \end{align}
   \end{lem}
   \begin{proof}
Since the proofs of \eqref{impo-ss-0} and \eqref{impo-ss}
are similar, we justify only \eqref{impo-ss}. Suppose
$T_1^2 + T_2^2$ is invertible. If $\delta \in (0,\infty)$
is such that $\sigma(T_1, T_2) \subseteq \rbb^2 \setminus
\delta\cdot \dbb,$ then by Theorem~\ref{spmt}(ii) with
$\psi(x_1,x_2) = x_1^2 + x_2^2$, we have
   \begin{align*}
\sigma(T_1^2 + T_2^2)=\sigma(\psi(T_1, T_2)) =
\psi(\sigma(T_1, T_2)) \subseteq [\delta^2, \infty),
   \end{align*}
which implies that $T_1^2 + T_2^2$ is invertible and
   \begin{align*}
\delta^2 \Le \min \sigma(T_1^2 + T_2^2)
\overset{\eqref{font-1}}= \|(T_1^2 + T_2^2)^{-1}\|^{-1}.
   \end{align*}
Reversing the argument with $\delta=\|(T_1^2 +
T_2^2)^{-1}\|^{-1/2}$, we obtain the converse implication
and~\eqref{impo-ss}. This completes the proof.
   \end{proof}
   We now describe the Taylor spectrum of an orthogonal sum
of pairs of commuting selfadjoint operators.
   \begin{pro}\label{sum0or-1}
Suppose that for every $n\in \nbb,$ $(T_{1,n},
T_{2,n})$ is a pair of commuting selfadjoint operators
on a nonzero complex Hilbert space $\hh_{n}.$ For
$j=1,2,$ let $T_j=\bigoplus_{n=1}^{\infty} T_{j,n}.$
Then $(T_1,T_2)$ is a pair of commuting selfadjoint
operators such~that
   \begin{align} \label{Il-pierze}
\sigma(T_1,T_2) = \overline{\bigcup_{n=1}^{\infty}
\sigma(T_{1,n},T_{2,n})}.
   \end{align}
   \end{pro}
   \begin{proof}
Set $\boldsymbol T = (T_1,T_2)$ and $\boldsymbol T_n =
(T_{1,n},T_{2,n})$ for $n\in \nbb.$ Denote by
$G_{\boldsymbol T}$ and $G_{\boldsymbol T_n}$ the joint
spectral measures of $\boldsymbol T$ and $\boldsymbol T_n,$
respectively. Let $G_{T_j}$ and $G_{T_{j,n}}$ be the
spectral measures of $T_j$ and $T_{j,n},$ respectively,
where $j=1,2$ and $n\in \nbb$. It is clear that
   \begin{align*}
G_{T_j} (\varDelta) = \bigoplus_{n=1}^{\infty}
G_{T_{j,n}}(\varDelta), \quad \varDelta \in \borel{\rbb},
\, j=1,2.
   \end{align*}
This implies that
   \begin{align*}
G_{T_1}(\varDelta_1) G_{T_2}(\varDelta_2) =
\bigoplus_{n=1}^{\infty} G_{\boldsymbol
T_n}(\varDelta_1 \times \varDelta_2), \quad
\varDelta_1, \varDelta_2 \in \borel{\rbb}.
   \end{align*}
Combined with the uniqueness of joint spectral
measures, this yields
   \begin{align} \label{gsumo-1}
G_{\boldsymbol T}(\varDelta) =
\bigoplus_{n=1}^{\infty} G_{\boldsymbol
T_n}(\varDelta), \quad \varDelta \in \borel{\rbb^2}.
   \end{align}
In view of Theorem~\ref{spmt}(i), it suffices to show
that
   \begin{align} \label{gsumo-2}
\supp{G}_{\boldsymbol T} =
\overline{\bigcup_{n=1}^{\infty} \supp G_{\boldsymbol
T_n}}.
   \end{align}
For this, take $(s,t) \in \rbb^2.$ If $(s,t) \notin
\supp{G}_{\boldsymbol T},$ then there exists an open
set $\varDelta$ in $\rbb^2$ such that $(s,t)\in
\varDelta$ and $G_{\boldsymbol T}(\varDelta)=0.$
Therefore by \eqref{gsumo-1}, $G_{\boldsymbol
T_n}(\varDelta)=0$ for all $n\in \nbb,$ which implies
that $(s,t) \notin \supp{G}_{\boldsymbol T_n}$ for all
$n\in \nbb.$ As a consequence, $\bigcup_{n=1}^{\infty}
\supp G_{\boldsymbol T_n} \subseteq
\supp{G}_{\boldsymbol T},$ which shows that the right
side of \eqref{gsumo-2} is contained in the left side.
In turn, if $(s,t) \notin
\overline{\bigcup_{n=1}^{\infty} \supp G_{\boldsymbol
T_n}},$ then there exists an open set $\varDelta$ in
$\rbb^2$ such that $(s,t)\in \varDelta$ and $\varDelta
\cap \overline{\bigcup_{n=1}^{\infty} \supp
G_{\boldsymbol T_n}}=\emptyset.$ Hence,
$G_{\boldsymbol T_n}(\varDelta)=0$ for all $n\in
\nbb,$ which together with \eqref{gsumo-1} implies
that $G_{\boldsymbol T}(\varDelta)=0.$ As a
consequence, $(s,t) \notin \supp{G}_{\boldsymbol T}.$
This completes the proof.
   \end{proof}
   \begin{cor} \label{arva-1}
If $\varGamma$ is an arbitrary nonempty compact subset of
$\rbb^2$ $($resp.\ $\rbb_+^2$$)$ and $\hh$ is a separable
infinite dimensional complex Hilbert space, then there
exists a pair $(T_1,T_2)$ of commuting selfadjoint
$($resp.\ positive$)$ operators $T_1,T_2\in \ogr{\hh}$ such
that $\varGamma=\sigma(T_1,T_2).$
   \end{cor}
   \begin{proof}
Since $\rbb^2$ is separable metric space, so is
$\varGamma.$ Hence, there exists a sequence
$\{(x_{1,n},x_{2,n})\}_{n=1}^{\infty} \subseteq \varGamma$
which is dense in $\varGamma.$ The proof is completed by
applying Proposition~\ref{sum0or-1} to $\hh_n=\cbb,$
$T_{1,n}= x_{1,n} I_{\cbb}$ and $T_{2,n}= x_{2,n} I_{\cbb}$
and by observing that according to \eqref{tplus0},
$\sigma(T_{1,n},T_{2,n}) = \{(x_{1,n},x_{2,n})\}$ for all
$n\in\nbb.$
   \end{proof}
   \begin{rem} \label{many-seu-2}
 A closer inspection of the proof reveals that
Proposition~\ref{sum0or-1} remains valid for families
(of arbitrary cardinality) of pairs of commuting
normal operators. As a consequence,
Corollary~\ref{arva-1} remains true if selfadjoint
operators are replaced by normal operators and $\rbb$
by $\cbb.$ What is more, using only the definition of
the Taylor spectrum, one can show that
\eqref{Il-pierze} holds (certainly without the
closure) for any finite number of pairs of commuting
operators (cf.\ \cite{Cur80}).
   \hfill $\diamondsuit$
   \end{rem}
   \section{\label{Sec2}Fundamental properties of operators of class $\gqb$}
   In this section we prove some basic properties of
operators of class $\gqb$ that are needed in this
paper. We begin by showing that the operators of class
$\gqb$ form a huge class which can be parameterized in
a sense by arbitrary pairs of commuting positive
operators.
   \begin{pro} \label{mem-ber-1}
Let $\hh=\hh_1 \oplus \hh_2$ be a nontrivial
orthogonal decomposition of a complex Hilbert space
$\hh.$ Then the following assertions are valid{\em :}
   \begin{enumerate}
   \item[(i)] if $T =  \big[\begin{smallmatrix} V & E \\
0 & Q \end{smallmatrix}\big] \in \gqbh,$ then $|Q|$
and $|E|$ are commuting positive operators such that
$\dim{\overline{\ob{|E|}}} \Le \dim{\ob{V}^{\perp}},$
   \item[(ii)] if $V\in\ogr{\hh_1}$ is an isometry and
$A,B\in \ogr{\hh_2}$ are commuting positive operators
such that $\dim{\overline{\ob{B}}} \Le
\dim{\ob{V}^{\perp}},$ then there exists $E\in
\ogr{\hh_2,\hh_1}$ such that
$T =  \big[\begin{smallmatrix} V & E \\
0 & A \end{smallmatrix}\big] \in \gqbh$ and $|E|=B.$
   \end{enumerate}
Moreover, if $T =  \big[\begin{smallmatrix} V & E \\
0 & Q \end{smallmatrix}\big] \in \gqbh,$ then there
exists $\widetilde E \in \ogr{\hh_2,\hh_1}$ such that
$\big[\begin{smallmatrix} V & \widetilde E \\
0 & |Q| \end{smallmatrix}\big] \in \gqbh$ and
$|\widetilde E| = |E|.$
   \end{pro}
   \begin{proof}
(i) That $|Q|$ and $|E|$ commute follows from \eqref{gqb-3}
and the square root theorem. Let $E=U|E|$ be the polar
decomposition of $E.$ Then $U$ maps ${\overline{\ob{|E|}}}$
unitarily onto $\overline{\ob{E}}.$ Since by \eqref{gqb-2},
$\overline{\ob{E}} \subseteq \ob{V}^\perp ,$ we are done.

(ii) Since $\dim{\overline{\ob{B}}} \Le
\dim{\ob{V}^{\perp}},$ there exist a closed subspace
$\mathcal M$ of $\ob{V}^\perp$ such that $\dim
\overline{\ob{B}} = \dim \mathcal M.$ Let $U \in
\ogr{\hh_2,\hh_1}$ be a unique partial isometry with the
initial space $\overline{\ob{B}}$ and the final space
$\mathcal M.$ Define $E\in \ogr{\hh_2,\hh_1}$ by $E=UB.$
Since $U^*U$ is the orthogonal projection of $\hh_2$ onto
the initial space $\overline{\ob{B}}$ of $U,$ we get
   \begin{align*}
|E|^2=E^*E = B (U^*U)B = B^2.
   \end{align*}
By the uniqueness of the square root, we deduce that
$|E|=B$. It is easily seen that $T =
\big[\begin{smallmatrix} V & E \\
0 & A \end{smallmatrix}\big] \in \gqbh.$

The ``moreover'' part is a direct consequence of (i)
and (ii). This completes the proof.
   \end{proof}
   \begin{cor} \label{bielan-1}
Suppose that $\hh_2$ is a nonzero complex Hilbert
space and $A,B\in \ogr{\hh_2}$ are commuting positive
operators. Then there exist a nonzero complex Hilbert
space $\hh_1,$ an isometry $V\in \ogr{\hh_1}$ and an
operator $E\in \ogr{\hh_2, \hh_1}$ such that $T =
\big[\begin{smallmatrix} V & E \\
0 & A \end{smallmatrix}\big] \in \gqbh$ $($relative to
$\hh=\hh_1\oplus \hh_2$$)$ and $|E|=B.$
   \end{cor}
   \begin{proof}
If $B=0,$ then we can apply
Proposition~\ref{mem-ber-1}(ii) to any nonzero complex
Hilbert space $\hh_1$ and an arbitrary isometry $V\in
\ogr{\hh_1}.$ In turn, if $B\neq 0,$ then we can take
an infinite dimensional complex Hilbert space $\hh_1$
such that $\dim{\overline{\ob{B}}} \Le \dim{\hh_1}.$
Then there exists an isometry $V\in \ogr{\hh_1}$ such
that
   \begin{align*}
\dim{\overline{\ob{B}}} \Le \dim{\hh_1} =
\dim{\ob{V}^{\perp}}.
   \end{align*}
Applying Proposition~\ref{mem-ber-1}(ii) completes the
proof.
   \end{proof}
   The theorem below is crucial for our further
investigations because the overwhelming majority of
results of this paper are stated in terms of the
Taylor spectrum of the pair $(|Q|,|E|).$
   \begin{thm} \label{bielan-2}
Suppose that $\varGamma$ is an arbitrary nonempty compact
subset of $\rbb_+^2$ and $\hh_2$ is a separable infinite
dimensional complex Hilbert space. Then there exists a
nonzero complex Hilbert
space $\hh_1$ and $T =  \big[\begin{smallmatrix} V & E \\
0 & Q \end{smallmatrix}\big] \in \gqbh$ $($relative to
$\hh=\hh_1\oplus \hh_2$$)$ such that $\sigma(|Q|,|E|) =
\varGamma.$
   \end{thm}
   \begin{proof}
It follows from Corollary~\ref{arva-1} that there exists a
pair $(A,B)$ of commuting positive operators
$A,B\in\ogr{\hh_2}$ such that $\sigma(A,B) = \varGamma.$
Applying Corollary~\ref{bielan-1} completes the proof.
   \end{proof}
   As shown below the norm of an operator of class $\gqb$
can be expressed in terms of the geometric spectral radius
of the pair $(|Q|,|E|)$.
   \begin{pro}\label{ogr-1}
Suppose that $T =  \big[\begin{smallmatrix} V & E \\
0 & Q \end{smallmatrix}\big] \in \gqbh.$ Then
   \begin{align} \label{norm-u}
\|T\| =\max\{1, r(|Q|,|E|)\}.
   \end{align}
   \end{pro}
   \begin{proof}
Let $G$ be the joint spectral measure of the pair
$(|Q|,|E|)$ and let ${G\mbox{-}\mathrm{ess\,sup\,}}
\varphi$ stand for the essential supremum of a Borel
function $\varphi\colon \rbb_+^2 \to \rbb_+$ with respect
to the measure $G.$ It follows from Definition~\ref{defq}
that
   \begin{align} \label{kaka-1}
T^*T=\begin{bmatrix}  I & 0 \\
0 & Q^*Q + E^*E \end{bmatrix}.
   \end{align}
Combined with the hypothesis that the spaces $\hh_1$
and $\hh_2$ are nonzero, this implies~that
   \allowdisplaybreaks
   \begin{align*}
\|T\|=\||T|\| &= \left\|\begin{bmatrix} I & 0 \\
0 & (|Q|^2 + |E|^2)^{1/2}\end{bmatrix}\right\|
   \\
& = \max\left\{1, \|(|Q|^2 + |E|^2)^{1/2}\|\right\}
   \\
& = \max\Big\{1, \underset{(s,t)\in\rbb_+^2}
{G\mbox{-}\mathrm{ess\,sup\,}} (s^2 + t^2)^{1/2}
\Big\}
   \\
& \hspace{-.4ex}\overset{(*)}= \max\Big\{1,
\max_{(s,t)\in\sigma(|Q|,|E|)} (s^2 + t^2)^{1/2}\Big\},
   \\
& = \max\{1, r(|Q|,|E|)\},
   \end{align*}
where $(*)$ follows from Theorem~\ref{spmt}(i) and the
continuity of the function $(s,t) \mapsto (s^2 +
t^2)^{1/2}$ on $\rbb_+^2.$ This completes the proof.
   \end{proof}
   \begin{rem} \label{Wik-2}
It follows from Proposition~\ref{ogr-1} that if $T\in
\ogr{\hh}$ is of class $\gqb$ and $\|T\|>1$, then the
geometric spectral radius $r(|Q|,|E|)$ does not depend on
the choice of an orthogonal decomposition $\hh=\hh_1\oplus
\hh_2$ of $\hh$ relative to which $T$ has a block matrix
representation \eqref{brep} with entries $V,$ $E$ and $Q$
satisfying the conditions \eqref{gqb-1}--\eqref{gqb-4}. We
refer the reader to Example~\ref{wid-dmo-1} for a detailed
discussion of the question of the existence of different
orthogonal decompositions of the underlying Hilbert space
$\hh$ relative to which a given operator $T\in \ogr{\hh}$
is of class $\gqb.$
   \hfill $\diamondsuit$
   \end{rem}
Next we characterize contractive, isometric and expansive
operators of class $\gqb.$
   \begin{pro} \label{kontr-2}
Suppose $T =  \big[\begin{smallmatrix} V & E \\
0 & Q \end{smallmatrix}\big] \in \gqbh.$ Then the
following conditions are equivalent{\em :}
   \begin{enumerate}
   \item[(i)] $T$ is a contraction $($resp., an isometry, an expansion$)$,
   \item[(ii)] $(|Q|,|E|)$ is a spherical contraction $($resp., a spherical
isometry, a spherical expansion$)$,
   \item[(iii)] $\sigma(|Q|,|E|) \subseteq \dbbc_+$
$\big($resp., $\sigma(|Q|,|E|) \subseteq \tbb_+$,
$\sigma(|Q|,|E|) \subseteq \rbb_+^2 \setminus
\dbb_+$$\big).$
   \end{enumerate}
Moreover, if $T$ is a contraction, then $\|T\|=1.$
   \end{pro}
   \begin{proof}
By Proposition~\ref{mem-ber-1}(i), $(|Q|,|E|)$ is a pair of
commuting positive operators. The equivalence
(i)$\Leftrightarrow$(ii) follows from \eqref{kaka-1}. Next
by applying Theorem~\ref{spmt}(ii) to the polynomial
$\psi(x_1,x_2) = x_1^2 + x_2^2,$ we get
   \begin{align*}
\psi(\sigma(|Q|,|E|)) = \sigma(|Q|^2 + |E|^2),
   \end{align*}
which together with \eqref{font-1} and
$\sigma(|Q|,|E|) \subseteq \rbb_+^2$ yields the
equivalence (ii)$\Leftrightarrow$(iii).

The ``moreover'' part is a direct consequence of
Proposition~\ref{ogr-1}.
   \end{proof}
For self-containedness, we state the following result
whose straightforward proof is left to the reader.
   \begin{pro}
The class $\gqb$ is closed under the operation of
taking orthogonal sums, i.e., if
$\{T_{\iota}\}_{\iota\in J}$ is a uniformly bounded
family of operators of class $\gqb,$ then
$\bigoplus_{\iota\in J} T_{\iota}$ is an operator of
class $\gqb.$
   \end{pro}
The following lemma provides a sufficient condition for the
product of two quasinormal operators to be quasinormal.
   \begin{lem} \label{JSZ1}
Suppose that $Q_1, Q_2 \in \ogr{\hh}$ are commuting
quasinormal operators such that $Q_1$ commutes with
$Q_2^*Q_2$ and $Q_2$ commutes with $Q_1^*Q_1$. Then
$Q_1Q_2$ is quasinormal. Moreover, any positive
integer power of a quasinormal operator is
quasinormal.
   \end{lem}
   \begin{proof}
We leave the simple algebraic proof of the first part
to the reader. The ``moreover'' part follows from the
first part by applying the formula
   \begin{align} \label{dydy-2}
Q^{*n}Q^n=(Q^*Q)^n, \quad n\in \zbb_+,
   \end{align}
which is valid for any quasinormal operator $Q$.
   \end{proof}
Our next goal is to give a sufficient condition for
the product of two operators of class $\gqb$ to be of
class $\gqb.$
   \begin{pro}
Suppose  $T_1 =  \big[\begin{smallmatrix} V_1 & E_1 \\
0 & Q_1 \end{smallmatrix}\big] \in \gqbh$ and
$T_2 =  \big[\begin{smallmatrix} V_2 & E_2 \\
0 & Q_2 \end{smallmatrix}\big] \in \gqbh$ are such
that
   \begin{align*}
\text{$Q_kQ_l^*Q_l = Q_l^*Q_lQ_k$ and $Q_kE_l^*E_l =
E_l^*E_lQ_k$ for all distinct $k,l\in\{1,2\}$.}
   \end{align*}
Then $T_1T_2 =  \big[\begin{smallmatrix} V  & E \\
0 & Q\end{smallmatrix}\big] \in \gqbh$, where
$V=V_1V_2$, $E=V_1E_2 + E_1Q_2$ and $Q=Q_1Q_2$.
   \end{pro}
   \begin{proof}
First notice that
   \begin{align*}
T_1 T_2 = \begin{bmatrix} V_1 V_2  & V_1E_2 + E_1 Q_2 \\
0 & Q_1 Q_2 \end{bmatrix}.
   \end{align*}
Clearly, $V_1V_2$ is an isometry, while by
Lemma~\ref{JSZ1}, $Q_1Q_2$ is a quasinormal operator.
Routine computations show that $(V_1V_2)^* (V_1E_2 +
E_1 Q_2) = 0$ and $Q_1Q_2$ commutes with $(V_1E_2 +
E_1 Q_2)^*(V_1E_2 + E_1 Q_2)$ meaning that $T_1T_2$ is
of class $\gqb$.
   \end{proof}
It turns out that the operation of taking positive integer
powers is inner in the class $\gqb.$ The class $\gqb$ is
also closed under the operation of taking the Cauchy dual.
Furthermore, we discuss the questions of when an operator
of class $\gqb$ is $\triangle_T$-regular and when it
satisfies the kernel condition introduced recently in
\cite{A-C-J-S}.
   \begin{pro} \label{babuc}
Suppose $T =  \big[\begin{smallmatrix} V & E \\
0 & Q \end{smallmatrix}\big] \in \gqbh.$ Then
   \begin{enumerate}
   \item[(i)]
$T^n =  \Big[\begin{smallmatrix} V^n  & E_n \\[.3ex ]
0 & Q^n \end{smallmatrix}\Big] \in \gqbh$ for any $n\in
\zbb_+,$ where
   \begin{align}  \label{dydy-1}
E_n =
   \begin{cases}
0 & \text{ if } n=0,
   \\
\sum_{j=1}^{n} V^{j-1}EQ^{n-j} & \text{ if } n\Ge 1,
   \end{cases}
   \end{align}
   \item[(ii)] $T^{*n}T^n =  \big[\begin{smallmatrix}  I & 0 \\
0 & \varOmega_n \end{smallmatrix}\big] \in \gqbh$ for
any $n\in \zbb_+,$ where
   \begin{align} \label{dudu-1}
\varOmega_n =
   \begin{cases}
I & \text{ if } n=0,
   \\
E^*E\Big(\sum_{j=0}^{n-1} (Q^*Q)^j\Big) + (Q^*Q)^n &
\text{ if } n\Ge 1,
   \end{cases}
   \end{align}
   \item[(iii)] $T$ is left-invertible if and only if
$\varOmega_1$ is invertible, or equivalently there exists
$\delta \in (0,\infty)$ such that $\sigma(|Q|, |E|)
\subseteq \rbb_+^2 \setminus \delta \cdot \dbb_+${\em ;} if
this is the case,~then
   \begin{align} \label{ale-num}
\max\big\{\delta\in \rbb_+\colon \sigma(|Q|, |E|) \subseteq
\rbb_+^2 \setminus \delta \cdot
\dbb_+\big\}=\|\varOmega_1^{-1}\|^{-1/2},
   \end{align}
   \item[(iv)] if $T$ is left-invertible, then
$T'\in \gqbh$ and
   \begin{align} \label{tipr-1}
T'= \begin{bmatrix} V  & E \varOmega_1^{-1} \\
0 & Q\varOmega_1^{-1} \end{bmatrix},
   \end{align}
   \item[(v)] $T$ is $\triangle_T$-regular if and only
if $T$ is an expansion,
   \item[(vi)] $T$ satisfies the kernel condition, i.e.,
$T^*T \jd{T^*} \subseteq \jd{T^*},$ if and only if $(|Q|^2
+ |E|^2-I)E^* h_1=0$ for every $h_1 \in \jd{V^*}$ such that
$E^*h_1 \in \ob{Q^*}.$
   \end{enumerate}
   \end{pro}
   \begin{proof}
(i) Using induction, one can verify that
   \begin{align*}
T^n =  \begin{bmatrix} V^n  & E_n \\
0 & Q^n \end{bmatrix}, \quad n \in \zbb_+,
   \end{align*}
where
   \begin{align} \label{dydy-3}
\text{$E_0=0$ and $E_{n+1}= VE_{n} + E Q^n$ for
$n\in\zbb_+$.}
   \end{align}
By induction, \eqref{dydy-3} implies \eqref{dydy-1}.
Clearly for any $n\in \zbb_+$, $V^n$ is an isometry
and, by Lemma~\ref{JSZ1}, $Q^n$ is a quasinormal
operator. Since $V$ is an isometry, we infer from
\eqref{gqb-2} and \eqref{dydy-1} that $V^{*n} E_n=0$
for any $n\in \zbb_+$. Employing \eqref{dydy-3}, we
see that
   \begin{align} \notag
E_{n+1}^* E_{n+1} & \overset{\eqref{gqb-1}\&
\eqref{gqb-2}}= E_{n}^* E_{n} + Q^{*n}E^*E Q^n
   \\ \label{dydy-4}
&\hspace{-.5ex}\overset{\eqref{gqb-3}\& \eqref{dydy-2}}=
E_{n}^* E_{n} + (Q^*Q)^n E^*E, \quad n\in \zbb_+.
   \end{align}
Using induction and \eqref{gqb-3}, we deduce that $Q$
commutes with $E_n^*E_n$ for all $n\in \zbb_+$. This
implies that $T^n$ is of class $\gqb$ for any $n\in
\zbb_+$.

(ii) It follows from (i) and \eqref{dydy-2} that
   \begin{align}  \label{dydy-5}
T^{*n} T^n = \begin{bmatrix} I  & 0 \\
0 & E_n^*E_n + (Q^*Q)^n \end{bmatrix}, \quad n\in
\zbb_+.
   \end{align}
Using induction, \eqref{dydy-4} and \eqref{gqb-3}, we
conclude that
   \begin{align*}
E_n^*E_n = E^*E\sum_{j=0}^{n-1} (Q^*Q)^j, \quad
n\in \nbb.
   \end{align*}
Combined with \eqref{dydy-5}, this yields (ii).

(iii) It is clear that $T$ is left-invertible if and only
if $T^*T$ is invertible, which by (ii) with $n=1$ is
equivalent to the invertibility of $\varOmega_1.$ The
remaining statement in (iii) is a direct consequence of
Theorem~\ref{spmt}(i) and Lemma~\ref{inf-del}.

(iv) It is a routine matter to show that \eqref{tipr-1}
holds and then to verify that $T'$ is of class $\gqb.$

(v) The ``only if'' part is obvious. To prove the ``if''
part, notice that by (ii),
   \begin{align*}
\triangle_T = \begin{bmatrix} 0  & 0 \\
0 & \varOmega_1 - I \end{bmatrix}.
   \end{align*}
Since $T$ is an expansion, we see that $\varOmega_1 - I\Ge
0$ and
   \begin{align} \label{dudu-2}
\triangle_T^{1/2} = \begin{bmatrix} 0  & 0 \\
0 & (\varOmega_1 - I)^{1/2} \end{bmatrix}.
   \end{align}
Knowing that $Q$ commutes with $\varOmega_1$ and using the
square root theorem, we deduce that $Q$ commutes with
$(\varOmega_1 - I)^{1/2}$, and consequently by
\eqref{dudu-2}, $\triangle_T^{1/2} T \triangle_T^{1/2} =
\triangle_T T$, which means that $T$ is
$\triangle_T$-regular.

(vi) Since $T^*= \big[\begin{smallmatrix} V^* & 0  \\
E^* & Q^* \end{smallmatrix}\big],$ we easily verify that
   \begin{align}  \label{dziup-1}
\jd{T^*} = \{h_1 \oplus h_2 \in \hh\colon h_1 \in \jd{V^*}
\text{ and } E^*h_1 + Q^* h_2=0\}.
   \end{align}
To prove the ``if'' part, suppose that
   \begin{align} \label{Gyeon-gju}
\big(|Q|^2 + |E|^2-I\big)\big((E^*\jd{V^*}) \cap
\ob{Q^*}\big)=\{0\}.
   \end{align}
If $h_1\oplus h_2 \in \jd{T^*}$, then in view of
\eqref{dziup-1} and \eqref{Gyeon-gju}, we have
   \begin{align*}
E^* h_1 + Q^*(|Q|^2 + |E|^2) h_2 &\overset{\eqref{gqb-3}\&
\eqref{gqb-4}} = E^* h_1 + (|Q|^2 + |E|^2)Q^* h_2
   \\
& \hspace{2.2ex}= (I -|Q|^2 - |E|^2) E^* h_1 =0.
   \end{align*}
Hence by (ii) with $n=1$ and \eqref{dziup-1},
$T^*T(h_1\oplus h_2) \in \jd{T^*}$, which justifies the
``if'' part. The ``only if'' part goes by reversing the
above argument. This completes the proof.
   \end{proof}
   \begin{cor}
Suppose $T =  \big[\begin{smallmatrix} V & E \\
0 & Q \end{smallmatrix}\big] \in \gqbh$ satisfies the
kernel condition, $E \neq 0$ and $\ob{Q^*}=\hh_2.$ Then
   \begin{enumerate}
   \item[(i)] $1$ is an eigenvalue of $|Q|^2+|E|^2,$
   \item[(ii)] $\sigma(|Q|, |E|) \cap \tbb_+ \neq \emptyset.$
   \end{enumerate}
   \end{cor}
   \begin{proof}
(i) Suppose, on the contrary, that $1$ is not an eigenvalue
of the operator $|Q|^2+|E|^2.$ Then by
Proposition~\ref{babuc}(vi), $\jd{V^*} \subseteq \jd{E^*}.$
This implies that $\ob{E} \subseteq \ob{V}.$ Since by
\eqref{gqb-2}, $\ob{E} \subseteq \ob{V}^{\perp},$ we see
that $E=0,$ a contradiction.

(ii) By (i) and Theorem~\ref{spmt}(ii) applied to the
polynomial $\psi(x_1,x_2)=x_1^2+x_2^2,$ we have $1 \in
\sigma(\psi(|Q|,|E|)) = \psi(\sigma(|Q|,|E|)),$ so there
exists $(s,t)\in \sigma(|Q|,|E|) \subseteq \rbb_+^2$ such
that $\psi(s,t)=1,$ which completes the proof.
   \end{proof}
   \section{\label{Sec3}Moment theoretic necessities}
In this section we prove a series of lemmata
concerning Hamburger and Stieltjes moment problems
needed in subsequent sections of this paper. We state
some of them in a more general context, namely for the
multi-dimensional moment problems, because the proofs
are essentially the same.

Below we use the standard multi-index notation, that
is, if $d\in \nbb,$ $\alpha=(\alpha_1,\ldots,\alpha_d)
\in \zbb_+^d$ and $x=(x_1,\ldots,x_d) \in \rbb^d,$
then we write $x^{\alpha}=x_1^{\alpha_1} \cdots
x_d^{\alpha_d}.$ A complex Borel measure $\mu$ on
$\rbb^d$ is said to be {\em compactly supported} if
there is a compact subset $K$ of $\rbb^d$ such that
$|\mu|(\rbb^d \setminus K)=0,$ where $|\mu|$ denotes
the total variation measure of $\mu.$ We write
$\supp{\mu}$ for the closed support of a finite
positive Borel measure $\mu$ on $\rbb^d$ (the support
exists because such $\mu$ is automatically regular,
see \cite[Theorem~2.18]{Rud}). We say that a
multi-sequence $\{\gamma_{\alpha}\}_{\alpha\in
\zbb_+^d}\subseteq \rbb$ is a {\em Hamburger moment
multi-sequence} (or {\em Hamburger moment sequence} if
$d=1$) if there exists a positive Borel measure $\mu$
on $\rbb^d,$ called a {\em representing measure} of
$\{\gamma_{\alpha}\}_{\alpha\in \zbb_+^d},$ such~that
   \begin{align} \label{Ham}
\gamma_{\alpha} = \int_{\rbb^d} x^{\alpha} d\mu(x),
\quad \alpha\in \zbb_+^d.
   \end{align}
If such $\mu$ is unique, then
$\{\gamma_{\alpha}\}_{\alpha\in \zbb_+^d}$ is said to
be {\em determinate}. If \eqref{Ham} holds for some
positive Borel measure $\mu$ on $\rbb^d$ supported in
$\rbb_+^d,$ then $\{\gamma_{\alpha}\}_{\alpha\in
\zbb_+^d}$ is called a {\em Stieltjes moment
multi-sequence} (or {\em Stieltjes moment sequence} if
$d=1$).
   \begin{lem} \label{mompr2}
Let $d\in \nbb.$ Suppose that $\mu_1$ and $\mu_2$ are
compactly supported complex Borel measures on $\rbb^d$ such
that
   \begin{align*}
\int_{\rbb^d} x^{\alpha} d\mu_1(x) = \int_{\rbb^d}
x^{\alpha} d\mu_2(x), \quad \alpha\in \zbb_+^d.
   \end{align*}
Then $\mu_1=\mu_2.$
   \end{lem}
   \begin{proof}
Since $|\mu_1-\mu_2|(\varDelta) \Le |\mu_1|(\varDelta)
+ |\mu_2|(\varDelta)$ for all Borel subsets
$\varDelta$ of $\rbb^d,$ the complex Borel measure
$\mu:=\mu_1-\mu_2$ is compactly supported, that is
$\supp{|\mu|} \subseteq [-R,R]^d$ for some
$R\in\rbb_+,$ and
   \begin{align} \label{mompr3}
\int_{\rbb^d} p \, d\mu = 0, \quad p\in \cbb[x_1,
\ldots,x_d].
   \end{align}
Let $f$ be a continuous complex function on $\rbb^d$
vanishing at infinity. By the Stone-Weierstrass theorem,
there exists a sequence $\{p_n\}_{n=1}^{\infty} \subseteq
\cbb[x_1, \ldots,x_d]$ such~that
   \begin{align} \label{mompr4}
\lim_{n\to \infty} \sup_{x\in [-R,R]^d}
|f(x)-p_n(x)|=0.
   \end{align}
Since
   \begin{align*}
\Big|\int_{\rbb^d} f d\mu\Big| &
\overset{\eqref{mompr3}}= \Big|\int_{\rbb^d} (f
-p_n)d\mu\Big| \Le \int_{[-R,R]^d} |f -p_n|d|\mu|
   \\
& \hspace{.7ex} \Le |\mu|([-R,R]^d) \sup_{x\in
[-R,R]^d} |f(x)-p_n(x)|, \quad n\in \nbb,
   \end{align*}
we deduce from \eqref{mompr4} that $\int_{\rbb^d} f
d\mu=0.$ Applying \cite[Theorems~ 6.19 and 2.18]{Rud}
yields $\mu=0,$ or equivalently, $\mu_1=\mu_2.$
   \end{proof}
   \begin{rem}
Concerning Lemma \ref{mompr2}, it is worth mentioning that
any sequence $\{\gamma_n\}_{n=0}^{\infty}\subseteq \cbb$
has infinitely many representing complex measures. For
this, note that there exists a complex Borel measure $\rho$
on $\rbb$ such that (see \cite{Boa39,Pol38,Dur89})
   \begin{align*}
\gamma_n = \int_{\rbb} x^n d\rho(x), \quad n\in \zbb_+.
   \end{align*}
Let $\{s_n\}_{n=0}^{\infty}$ be an indeterminate Hamburger
moment sequence with two distinct representing measures
$\mu_1$ and $\mu_2$ (see \cite{Sti95,ber98}). Then $\mu:=
\mu_1-\mu_2$ is a signed Borel measure on $\rbb$ such that
   \begin{align*}
\int_{\rbb} x^n d\mu(x)=0, \quad n\in \zbb_+.
   \end{align*}
As a consequence, we have
   \begin{align*}
\gamma_n = \int_{\rbb} x^n d(\rho + \vartheta \mu)(x),
\quad n\in \zbb_+, \, \vartheta \in \cbb.
   \end{align*}
Moreover, the mapping $\cbb \ni \vartheta \longmapsto \rho
+ \vartheta \mu$ is an injection.
   \hfill $\diamondsuit$
   \end{rem}
   \begin{lem} \label{mompr5}
If $d\in \nbb,$ $R=(R_1, \ldots,R_d)\in \rbb_+^d$ and
$\mu$ is a complex Borel measure on $\rbb^d$ such that
$\supp{|\mu|} \subseteq [-R_1,R_1] \times \ldots
\times [-R_d,R_d],$ then
   \begin{align*}
\Big|\int_{\rbb^d} x^{\alpha} d\mu(x)\Big| \Le
|\mu|(\rbb^d) R^{\alpha}, \quad \alpha \in \zbb_+^d.
   \end{align*}
   \end{lem}
   \begin{proof}
Since $|x^{\alpha}| \Le R^{\alpha}$ for all $\alpha\in
\zbb_+^d$ and $x\in \supp{|\mu|}$, we get
   \begin{align*}
\Big|\int_{\rbb^d} x^{\alpha} d\mu(x)\Big| \Le
\int_{\rbb^d} |x^{\alpha}| \, d|\mu|(x) \Le
|\mu|(\rbb^d) R^{\alpha}, \quad \alpha \in \zbb_+^d.
&\hfill \qedhere
   \end{align*}
   \end{proof}
   \begin{lem} \label{mompr6}
Let $d\in\nbb,$ $\mu$ be a compactly supported complex
Borel measure on $\rbb^d$ and $\gamma_{\alpha}=
\int_{\rbb^d} x^{\alpha} d\mu(x)$ for $\alpha\in
\zbb_+^d.$ Then the following conditions are~
equivalent{\em :}
   \begin{enumerate}
   \item[(i)] $\{\gamma_{\alpha}\}_{\alpha\in \rbb^d}$
is a Hamburger moment multi-sequence,
   \item[(ii)] $\mu$ is a positive measure.
   \end{enumerate}
Moreover, if {\em (i)} holds, then
$\{\gamma_{\alpha}\}_{\alpha\in \rbb^d}$ is
determinate.
   \end{lem}
   \begin{proof}
(i)$\Rightarrow$(ii) Let $\nu$ be a representing measure of
$\{\gamma_{\alpha}\}_{\alpha\in\zbb_+^d}.$ By
Lemma~\ref{mompr5},
   \begin{align*}
\lim_{n\to \infty} \Big(\int_{\rbb^d} x_j^{2n} d
\nu(x)\Big)^{1/2n} = \lim_{n\to \infty} \Big|\int_{\rbb^d}
x_j^{2n} d \mu(x)\Big|^{1/2n} \Le R_j, \quad j=1,\ldots,d,
   \end{align*}
where $R_1, \ldots,R_d$ are as in Lemma~ \ref{mompr5}.
Thus, by \cite[Exercise~ 4(e), p.~ 71]{Rud} (see also
\cite[Problem~1(a), p.\ 332]{Sim-4}), $\supp{\nu} \subseteq
[-R_1,R_1] \times \ldots \times [-R_d,R_d]$. Hence by
Lemma~ \ref{mompr2}, $\{\gamma_{\alpha}\}_{\alpha\in
\rbb^d}$ is determinate, $\mu=\nu$ and so $\mu$ is a
positive measure.

   The implication (ii)$\Rightarrow$(i) is trivial.
   \end{proof}
We state now the following fact which we need in the
proof of Lemma~\ref{mompr7}. It can be proved by
induction on the degree of the polynomial in question.
   \begin{lem}[\mbox{\cite[Exercise 7.2]{Dick}}] \label{Dick}
If $p\in \cbb[x]$ is of degree $k\in \zbb_+$, then
   \begin{align*}
(\triangle^m \check{p})_n = p^{(m)}(0), \quad n\in
\zbb_+,\, m \Ge k,
   \end{align*}
where $\triangle\colon \cbb^{\zbb_+} \to
\cbb^{\zbb_+}$ is the linear transformation given by
$(\triangle \gamma)_n = \gamma_{n+1}-\gamma_n$ for
$n\in \zbb_+$ and $\gamma \in \cbb^{\zbb_+},$
$\check{p}\in \cbb^{\zbb_+}$ is given by
$\check{p}_n=p(n)$ for $n\in \zbb_+$ and $p^{(m)}(0)$
stands for the $m$th derivative of $p$ at $0.$
   \end{lem}
As shown below, a nonconstant polynomial perturbation of a
Hamburger moment sequence is never a Hamburger moment
sequence.
   \begin{lem} \label{mompr7}
Let $\{\gamma_n\}_{n=0}^{\infty}$ be a Hamburger
moment sequence having a compactly supported
representing measure $\mu$ and let $p\in \rbb[x].$
Then the following conditions are equivalent{\em :}
   \begin{enumerate}
   \item[(i)] the sequence $\{\gamma_n + p(n)\}_{n=0}^{\infty}$ is
a Hamburger moment sequence,
   \item[(ii)] $p$ is a constant polynomial and $\mu(\{1\}) + p(0)
\Ge 0.$
   \end{enumerate}
Moreover, if {\em (ii)} holds, then $\mu+p(0)\delta_1$
is a compactly supported representing measure of
$\{\gamma_n + p(n)\}_{n=0}^{\infty}$.
   \end{lem}
   \begin{proof}
Without loss of generality we may assume that the
polynomial $p$ is nonzero, that is $k:=\deg p \Ge 0.$

   (i)$\Rightarrow$(ii) Define $\{\tilde
\gamma_n\}_{n=0}^{\infty}$ by
   \begin{align} \label{ti-du}
\tilde \gamma_n=\gamma_n + p(n), \quad n\in \zbb_+.
   \end{align}
Let $\tilde \mu$ be a representing measure of $\{\tilde
\gamma_n\}_{n=0}^{\infty}.$ Applying Lemma~\ref{mompr5} to
$\{\gamma_n\}_{n=0}^{\infty}$ and using the fact that
$\sup_{n\in \zbb_+}|p(n)|e^{-n} < \infty$, we deduce that
the measure $\tilde\mu$ is compactly supported (see the
proof of Lemma~ \ref{mompr6}). Since, by Lemma~\ref{Dick},
$(\triangle^k \check{p})_n = p^{(k)}(0)$ for all $n\in
\zbb_+,$ applying $\triangle^k$ to both sides of
\eqref{ti-du} yields
   \begin{align*}
\int_{\rbb} x^n(x-1)^k d \tilde\mu(x) = \int_{\rbb}
x^n(x-1)^k d \mu(x) + p^{(k)}(0), \quad n \in \zbb_+.
   \end{align*}
Together with Lemma \ref{mompr2}, this implies that
   \begin{align} \label{Upart-Sam}
\int_{\varDelta} (x-1)^k d \tilde\mu(x) =
\int_{\varDelta} (x-1)^k d \mu(x) + p^{(k)}(0)
\delta_1(\varDelta), \quad \varDelta \in \borel{\rbb}.
   \end{align}
If $k\Ge 1,$ then by substituting $\varDelta=\{1\},$
we get $p^{(k)}(0)=0,$ which gives a contradiction.
Therefore, $p$ must be a constant polynomial.
Substituting $k=0$ into \eqref{Upart-Sam}, we get
(ii).

The implication (ii)$\Rightarrow$(i) and the
``moreover'' part are obvious.
   \end{proof}
The following is an immediate consequence of
Lemma~\ref{mompr7} applied to $\gamma_n=0$ and
$\mu=0$.
   \begin{lem} \label{mompr10}
For $p\in \rbb[x],$ the following conditions are
equivalent{\em :}
   \begin{enumerate}
   \item[(i)] $\{p(n)\}_{n=0}^{\infty}$ is a Hamburger
moment sequence,
   \item[(ii)] $\{p(n)\}_{n=0}^{\infty}$ is a
Stieltjes moment sequence,
   \item[(iii)] $p$ is  a constant polynomial and $p(0) \Ge 0.$
   \end{enumerate}
   \end{lem}
   \begin{rem}
The implication (i)$\Rightarrow$(iii) of
Lemma~\ref{mompr10} can be proved more directly. Let
$\mu$ be a representing measure of
$\{p(n)\}_{n=0}^{\infty}.$ Clearly, $p(0) = \mu(\rbb)
\Ge 0.$ Suppose, on the contrary, that $k:=\deg p \Ge
1.$ By the Schwarz inequality, we have
   \begin{align*}
p(n)^2 = \Big(\int_{\rbb} x^0 x^n d\mu(x)\Big)^2 \Le
\int_{\rbb} x^0 d\mu(x) \int_{\rbb} x^{2n} d\mu(x) =
p(0)p(2n), \quad n \in \zbb_+.
   \end{align*}
Denote by $a$ the leading coefficient of $p.$ The
above inequality implies that
   \begin{align*}
a^2 = \lim_{n\to\infty} \frac{p(n)^2}{n^{2k}} \Le
\lim_{n \to \infty} \frac{p(0)p(2n)}{n^{2k}} =0,
   \end{align*}
which contradicts the fact that $a \neq 0.$ Therefore, $p$
is a constant polynomial.
   \hfill $\diamondsuit$
   \end{rem}
For the sake of completeness, we provide a proof of the
following lemma which will be used in subsequent parts of
this paper.
   \begin{lem} \label{main-ad}
Let $G\colon\borel{X} \to \ogr{\hh}$ be a regular spectral
measure on a topological Hausdorff space $X$ with compact
support, $\varphi\colon X \to \cbb$ be a continuous
function and $\varSigma$ be a relatively open subset of
$\supp{G}.$ Then the spectral integral $\int_{\varSigma}
\varphi d G,$ which is a bounded operator, is positive if
and only if $\varSigma \subseteq \{x\in X\colon \varphi(x)
\Ge 0\}.$
   \end{lem}
   \begin{proof}
Since $\sup_{x \in \supp{G}}|\varphi(x)| < \infty,$
$\int_{\varSigma} \varphi d G \in \ogr{\hh}.$ To prove the
``only if'' part, assume that $\int_{\varSigma} \varphi d G
\Ge 0.$ Then $\int_{\varSigma} \varphi(x) \is{G(d x)h}{h}
\Ge 0$ for all $h\in \hh.$ Substituting $G(\varDelta)h$ in
place of $h$ with $\varDelta\in \borel{\varSigma}$, we see
that $\int_{\varDelta} \varphi(x) \is{G(dx)h}{h} \Ge 0$ for
all $\varDelta \in \borel{\varSigma}$ and $h\in \hh.$
Combined with \cite[Theorem~1.6.11]{Ash00}, this implies
that $\big\langle G\big(K_{\varphi})h,h\big\rangle = 0$ for
all $h\in \hh,$ where $K_{\varphi}:=\{x\in \varSigma\colon
\varphi(x) \in \cbb \setminus \rbb_+\}.$ Since
$K_{\varphi}$ is a relatively open subset of $\supp{G}$ and
$G(K_{\varphi}) = 0$, we conclude that $K_{\varphi} =
\emptyset,$ which means that $\varSigma \subseteq \{x\in
X\colon \varphi(x) \Ge 0\}.$ The ``if'' part is obvious.
   \end{proof}
   \begin{lem} \label{wkw-integ}
Let $G\colon \borel{X} \to \ogr{\hh}$ be a regular
spectral measure on a topological Hausdorff space $X$
with compact support and let $\varphi_n\colon X \to
\rbb,$ $n\in \zbb_+,$ be continuous functions. Then
the following conditions are equivalent{\em :}
   \begin{enumerate}
   \item[(i)] $\{\varphi_n(x)\}_{n=0}^{\infty}$ is a
Stieltjes moment sequence for every $x\in \supp{G},$
   \item[(ii)] $\{\int_X\varphi_n(x) \is{G(d x)h}{h}\}_{n=0}^{\infty}$ is a
Stieltjes moment sequence for every $h\in \hh.$
   \end{enumerate}
   \end{lem}
   \begin{proof}
As in Lemma~\ref{main-ad}, $\int_{X} \varphi d G\in
\ogr{\hh}$ whenever $\varphi\colon X \to \cbb$ is
continuous.

(i)$\Rightarrow$(ii) This can be easily deduced from
\cite[Theorem~6.2.5]{B-C-R} (see also
\cite[Lemma~3.2]{A-C-J-S}).

(ii)$\Rightarrow$(i) Fix $n\in \zbb_+$ and
$\boldsymbol{\lambda}=(\lambda_0, \ldots, \lambda_n)
\in \cbb^{n+1}.$ Define the continuous function
$\varPhi_{\boldsymbol{\lambda}} \colon X \to \cbb$ by
   \begin{align*}
\varPhi_{\boldsymbol{\lambda}}(x) = \sum_{k,l=0}^n
\varphi_{k+l}(x) \lambda_k \bar\lambda_l, \quad x\in
X.
   \end{align*}
Applying the implication (iii)$\Rightarrow$(i) of
\cite[Theorem~6.2.5]{B-C-R}, we see that
   \begin{align*}
\int_{X} \varPhi_{\boldsymbol{\lambda}}(x) \is{G(d
x)h}{h} = \sum_{k,l=0}^n \int_{X} \varphi_{k+l}(x)
\is{G(d x)h}{h} \lambda_k \bar\lambda_l \Ge 0, \quad h
\in \hh.
   \end{align*}
Hence $\int_{X} \varPhi_{\boldsymbol{\lambda}} d G \Ge
0,$ so by Lemma~\ref{main-ad},
$\varPhi_{\boldsymbol{\lambda}}(x) \Ge 0$ for all
$x\in \supp{G},$ that is
   \begin{align*}
\sum_{k,l=0}^n \varphi_{k+l}(x) \lambda_k
\bar\lambda_l \Ge 0, \quad x\in \supp{G}.
   \end{align*}
A similar argument shows that
   \begin{align*}
\sum_{k,l=0}^n \varphi_{k+l+1}(x) \lambda_k
\bar\lambda_l \Ge 0, \quad x\in \supp{G}.
   \end{align*}
Finally, by applying the implication
(i)$\Rightarrow$(iii) of \cite[Theorem~6.2.5]{B-C-R},
we complete the proof.
   \end{proof}
Before concluding this section, we recall the
celebrated criterion for subnormality of bounded
operators essentially due to Lambert (see \cite{lam};
see also \cite[Proposition 2.3]{St1}).
   \begin{align} \label{Lam}
   \begin{minipage}{70ex}
{\em An operator $T\in \ogr{\hh}$ is subnormal if and
only if for every $h\in \hh,$ $\{\|T^n
h\|^2\}_{n=0}^{\infty}$ is a Stieltjes moment
sequence.}
   \end{minipage}
   \end{align}

The following general characterization of subnormal
operators fits nicely into the scope of the present
investigations. It will be used to provide the second
proof of Corollary~\ref{main-cor}.
   \begin{thm}  \label{dziw-char}
Suppose that $\varphi_n\colon X \to \rbb,$ $n\in
\zbb_+,$ are continuous functions on a topological
Hausdorff space $X$ of the form
   \begin{align} \label{char-mom-a}
\varphi_n(x) = \int_{\rbb_+} t^n d \mu_x(t), \quad n
\in \zbb_+, \, x \in X,
   \end{align}
where each $\mu_x$ is a compactly supported complex
Borel measure on $\rbb_+.$ Furthermore, assume that
$T\in \ogr{\hh}$ is an operator for which there exists
a regular spectral measure $G\colon \borel{X} \to
\ogr{\hh}$ with compact support such that
   \begin{align} \label{char-mom-b}
T^{*n}T^n = \int_X \varphi_n(x) G(d x), \quad n\in
\zbb_+.
   \end{align}
Then $T$ is subnormal if and only if $\mu_x$ is a
positive measure for every $x\in \supp{G}.$
   \end{thm}
   \begin{proof}
By \eqref{Lam} and \eqref{char-mom-b}, the operator $T$ is
subnormal if and only if the sequence $\{\int_X\varphi_n(x)
\is{G(d x)h}{h}\}_{n=0}^{\infty}$ is a Stieltjes moment
sequence for every $h\in \hh.$ By Lemma~\ref{wkw-integ},
the latter holds if and only if
$\{\varphi_n(x)\}_{n=0}^{\infty}$ is a Stieltjes moment
sequence for every $x\in \supp{G},$ which in view of
\eqref{char-mom-a} and Lemma~\ref{mompr6} is equivalent to
the fact that $\mu_x$ is a positive measure for every $x\in
\supp{G}.$
   \end{proof}
   \section{\label{Sec4}Proof of the main result and some consequences}
Before proving Theorem~\ref{main}, which is the main result
of this paper, we make the following useful observation
being a direct consequence of \eqref{tplus0} and
\eqref{Sam-Jan2Zen}.
   \begin{align} \label{luft-1}
   \begin{minipage}{70ex}
{\em If $T = \big[\begin{smallmatrix} V & E \\
0 & Q \end{smallmatrix}\big] \in \gqbh,$ then $E\neq 0$ if
and only if $\sigma_{\sharp}(|Q|,|E|) \neq \emptyset,$}
   \end{minipage}
   \end{align}
where $\sigma_{\sharp}(|Q|,|E|)=\sigma(|Q|,|E|) \cap
(\rbb_+ \times (0,\infty)).$
   \begin{proof}[Proof of Theorem~ \ref{main}]
(i)$\Leftrightarrow$(iii) In view of
Proposition~\ref{mem-ber-1}(i), $(|Q|,|E|)$ is a pair of
commuting positive operators. Let $G$ be the joint spectral
measure of $(|Q|,|E|).$ Then, by Theorem~\ref{spmt}(i) and
\cite[Theorem~2.18]{Rud}, the measure $G$ is compactly
supported and regular. It follows from \eqref{fr-fr-1} and
\eqref{dudu-1} that
   \begin{align} \label{slipy}
\varOmega_n = \int_{\rbb_+^2} \varphi_n d G, \quad
n\in \zbb_+,
   \end{align}
where $\varphi_n\colon \rbb_+^2 \to \rbb_+$ is the
continuous function defined by
   \begin{align} \label{form-1}
\varphi_n(s,t) =
   \begin{cases}
1 & \text{if } n=0,
   \\
t^2\big(\sum_{j=0}^{n-1} s^{2j}\big) + s^{2n} &
\text{if } n\Ge 1,
   \end{cases}
\quad (s,t) \in \rbb_+^2.
   \end{align}
Notice that by Proposition~\ref{babuc}(ii) and \eqref{Lam},
the operator $T$ is subnormal if and only if
$\{\is{\varOmega_n h}{h}\}_{n=0}^{\infty}$ is a Stieltjes
moment sequence for every $h\in \hh_2.$ Hence in view of
\eqref{slipy} and Lemma~\ref{wkw-integ}, $T$ is subnormal
if and only if $\supp{G} \subseteq \varXi,$ where $\varXi$
is the set of all points $(s,t)\in \rbb_+^2$ for which
$\{\varphi_n(s,t)\}_{n=0}^{\infty}$ is a Stieltjes moment
sequence. Therefore, according to Theorem~\ref{spmt}(i), to
get the equivalence (i)$\Leftrightarrow$(iii), it is enough
to show that $\varXi= \dbbc_+ \cup \big(\rbb_+ \times
\{0\}\big).$ For this purpose, take $(s,t)\in \rbb_+^2$ and
consider two cases.

{\sc Case 1.} $s=1.$

Then by \eqref{form-1}, we have $\varphi_n(s,t)= 1 + n
t^2.$ Applying Lemma~\ref{mompr10} to $p(x)= 1 + t^2 x$, we
see that $(1,t)\in \varXi$ if and only if $t=0.$

{\sc Case 2.} $s\neq 1.$

Then by \eqref{form-1} we have
   \begin{align} \label{mlask-1}
\varphi_n(s,t) = \frac{t^2}{1-s^2} +
\bigg(1-\frac{t^2}{1-s^2}\bigg)s^{2n}, \quad n \in
\zbb_+.
   \end{align}
This implies that
   \begin{align} \label{smi-rep1}
\varphi_n(s,t) = \int_{\rbb_+} x^n \mu_{s,t}(d x),
\quad n\in \zbb_+,
   \end{align}
where $\mu_{s,t} \colon \borel{\rbb_+} \to \rbb$ is the
signed measure of the form
   \begin{align} \label{smi-rep2}
\mu_{s,t} = \frac{t^2}{1-s^2} \delta_{1} +
\bigg(1-\frac{t^2}{1-s^2}\bigg) \delta_{s^2}.
   \end{align}
Using Lemma~\ref{mompr6}, we conclude that $(s,t)\in
\varXi$ if and only if the measure $\mu_{s,t}$ is
positive, or equivalently if and only if
   \begin{align} \label{surp-1}
0 \Le \frac{t^2}{1-s^2} \Le 1.
   \end{align}
If $t=0,$ then \eqref{surp-1} holds. If $t \neq 0,$ then
\eqref{surp-1} holds if and only if $(s,t) \in \dbbc_+.$
Thus $(s,t) \in \varXi$ if and only if $(s,t) \in
\big(\dbbc_+ \cup \big(\rbb_+ \times \{0\}\big)\big)
\setminus \{(1,0)\}.$

Summarizing Cases 1 and 2, we conclude that $\varXi=
\dbbc_+ \cup \big(\rbb_+ \times \{0\}\big),$ which gives
the desired equivalence (i)$\Leftrightarrow$(iii).

(ii)$\Leftrightarrow$(iii) This is obvious due to the fact
that $\sigma(|Q|,|E|) \subseteq \rbb_+^2$ (see
Theorem~\ref{spmt}(i)).

Before proving the equivalence (ii)$\Leftrightarrow$(iv),
we make necessary preparations. Set $K=\sigma(|Q|,|E|).$
Let $G_{|Q|}$ and $G_{|E|}$ be the spectral measures of
$|Q|$ and $|E|,$ respectively. Since $P$ is the orthogonal
projection of $\hh_2$ onto $\hh_2 \ominus \jd{|E|}$ and
$\jd{|E|}=\ob{G_{|E|}(\{0\})},$ we see that
   \begin{align}  \label{pge-1}
P=G_{|E|} ((0,\infty)).
   \end{align}
By Proposition~\ref{mem-ber-1}(i), $|Q|$ commutes with
$|E|$ so it commutes with $G_{|E|}.$ As a consequence,
the operators $|Q|,$ $|E|$ and $P$ commute. Combined
with Theorem~\ref{spmt}(i), this yields
   \begin{align} \notag
\int_{\sigma_{\sharp}(|Q|,|E|)} (s^2 + t^2) G(ds,dt) &=
\int_{K \cap \big(\rbb_+ \times (0,\infty)\big)} (s^2 +
t^2) G(ds,dt)
   \\  \notag
&= \int_{\rbb_+ \times (0,\infty)} (s^2 + t^2) G(ds,dt)
   \\  \notag
&= \int_{\rbb_+} s^2 G_{|Q|}(ds) G_{|E|} ((0,\infty)) +
\int_{(0,\infty)} t^2 G_{|E|}(dt)
   \\  \notag
&= |Q|^2 G_{|E|} ((0,\infty)) + |E|^2.
  \\  \label{sum-a2}
&= (|Q|P)^2 + |E|^2.
   \end{align}

(ii)$\Rightarrow$(iv) Suppose that (ii) holds. Then by
\eqref{sum-a2}, we have
   \begin{align*}
(|Q|P)^2 + |E|^2 = \int_{\sigma_{\sharp}(|Q|,|E|)} (s^2 +
t^2) G(ds,dt) \Le G(\sigma_{\sharp}(|Q|,|E|)) \Le I,
   \end{align*}
which means that $(|Q|P,|E|)$ is a spherical contraction.

(iv)$\Rightarrow$(ii) Suppose now that (iv) holds, i.e.,
$(|Q|P)^2 + |E|^2 \Le I.$ Since $I_{\hh_2} - P$ is the
orthogonal projection of $\hh_2$ onto $\jd{|E|},$ we deduce
that
   \begin{align} \label{dra-1-q}
(|Q|P)^2 + |E|^2 \Le P.
   \end{align}
Observe now that
   \begin{align*}
G(\sigma_{\sharp}(|Q|,|E|)) & = G\big(K \cap
\big(\rbb_+ \times (0,\infty)\big)\big)
   \\
& = G(\rbb_+ \times (0,\infty)) = G_{|E|}((0,\infty))
\overset{\eqref{pge-1}} = P.
   \end{align*}
Combined with \eqref{sum-a2} and \eqref{dra-1-q}, this leads to
   \begin{align*}
\int_{\sigma_{\sharp}(|Q|,|E|)} \big(1-(s^2 + t^2)\big)
G(ds,dt) \Ge 0.
   \end{align*}
Since $\sigma_{\sharp}(|Q|,|E|)$ is a relatively open
subset of $\sigma(|Q|,|E|),$ we infer from
Theorem~\ref{spmt}(i) and Lemma~\ref{main-ad} that
$\sigma_{\sharp}(|Q|,|E|) \subseteq \dbbc_+.$

(iv)$\Leftrightarrow$(v) That $\mcal=\overline{\ob{|E|}}$
reduces $|Q|$ and $|E|$ follows from the fact that $P$
commutes with $|Q|$ and $|E|.$ Combined with the equations
$(|Q|P)|_{\jd{|E|}} = |E|\big|_{\jd{|E|}} = 0,$ this leads
to the desired equivalence.

(v)$\Leftrightarrow$(vi) This equivalence can be
proved in the same way as the equivalence
(ii)$\Leftrightarrow$(iii) of
Proposition~\ref{kontr-2}.

The ``moreover'' part is a direct consequence of (iii) and
\eqref{Sam-Jan2Zen}. This completes the proof.
   \end{proof}
In the rest of this section we record some consequences of
Theorem~\ref{main}. We begin with the following corollary
which is immediate from Theorem~\ref{main}(v).
   \begin{cor} \label{gyeongju-c}
Suppose that $T = \big[\begin{smallmatrix} V & E \\
0 & Q \end{smallmatrix}\big] \in \gqbh$ and $\boldsymbol
z=(z_1, z_2, z_3)\in \cbb^3$ is such that $|z_1|=1$ and
$|z_j| \Le 1$ for $j=2,3.$ Then $T_{\boldsymbol
z}:=\Big[\begin{smallmatrix} z_1 V & z_2 E \\[.3ex] 0 & z_3 Q
\end{smallmatrix}\Big] \in \gqbh.$ Moreover, if $T$ is
subnormal, then so is $T_{\boldsymbol z}.$
   \end{cor}
The next corollary follows from Proposition~\ref{kontr-2}
and Theorem~\ref{main} (recall that by
Proposition~\ref{ogr-1} the contractions of class $\gqb$
are of norm $1$).
   \begin{cor}  \label{main-cor}
Any contraction of class $\gqb$ is subnormal.
   \end{cor}
   As shown below, Corollary~\ref{main-cor} can also be
   deduced from Theorem~\ref{dziw-char}.
   \begin{proof}[Second Proof of Corollary~\ref{main-cor}]
Assume that $T$ is a contraction. Let $\widetilde G$ be the
joint spectral measure of $(|Q|,|E|).$ Set $X=\dbbc_+.$ It
follows from Proposition~\ref{kontr-2} that
$\sigma(|Q|,|E|) \subseteq X.$ Hence, by
Theorem~\ref{spmt}(i), the function $G \colon \borel{X} \to
\ogr{\hh}$ defined by
   \begin{align*}
G(\varDelta) = \delta_{(1,0)}(\varDelta) I_{\hh_1} \oplus
\widetilde G(\varDelta), \quad \varDelta \in \borel{X},
   \end{align*}
is a spectral measure. In view of
Proposition~\ref{babuc}(ii) and \eqref{slipy}, the
condition \eqref{char-mom-b} holds with $\varphi_n$ as in
\eqref{form-1}. Moreover, by \eqref{smi-rep1} and
\eqref{smi-rep2}, the condition \eqref{char-mom-a} holds,
where $\mu_x$ is the positive Borel measure on $\rbb_+$
given by \eqref{smi-rep2} for $x=(s,t)\in X \setminus
\{(1,0)\}$ and $\mu_{(1,0)} = \delta_{1}.$ Hence, by
Theorem~\ref{dziw-char}, $T$ is subnormal.
   \end{proof}
Below we indicate two subclasses of $\gqb$ for which
subnormality is completely characterized by contractivity.
   \begin{cor} \label{dra-b}
Suppose $T =  \big[\begin{smallmatrix} V & E \\
0 & Q \end{smallmatrix}\big] \in \gqbh,$ where $E=\alpha
U,$ $\alpha\in \cbb \setminus \{0\}$ and $U\in
\ogr{\hh_2,\hh_1}$ is an isometry. Then the following
conditions are equivalent{\em :}
   \begin{enumerate}
   \item[(i)] $T$ is subnormal,
   \item[(ii)] $\|Q\|^2 + |\alpha|^2 \Le 1,$
   \item[(iii)] $T$ is a contraction.
   \end{enumerate}
   \end{cor}
   \begin{proof}
By \eqref{tplus0}, we have
   \begin{align} \label{fur-1}
\sigma(|Q|,|E|) = \sigma(|Q|,|\alpha|I_{\hh_2}) =
\sigma(|Q|) \times \{|\alpha|\}.
   \end{align}
Since
   \begin{align} \label{crak-1}
\max\sigma(|Q|) = \||Q|\| =\|Q\|,
   \end{align}
we deduce from \eqref{fur-1} that
   \begin{align} \label{fur-2}
\text{$\sigma(|Q|,|E|) \subseteq \dbbc_+$ if and only if
$\|Q\|^2 + |\alpha|^2 \Le 1.$}
   \end{align}

(i)$\Leftrightarrow$(ii) Using the assumption that
$\alpha\neq 0$ and applying Theorem~\ref{main}, we deduce
from \eqref{fur-1} and \eqref{fur-2} that $T$ is subnormal
if and only if $\|Q\|^2 + |\alpha|^2 \Le 1.$

(ii)$\Leftrightarrow$(iii) This is a direct consequence of
\eqref{fur-2} and Proposition~\ref{kontr-2}.
   \end{proof}
The following is a variant of Corollary~ \ref{dra-b} with
essentially the same proof.
   \begin{cor}
Suppose $T =  \big[\begin{smallmatrix} V & E \\
0 & Q \end{smallmatrix}\big] \in \gqbh,$ where $Q=\alpha
U,$ $\alpha\in \cbb$ and $U\in \ogr{\hh_2}$ is an isometry.
If $E\neq 0,$ then the following conditions are
equivalent{\em :}
   \begin{enumerate}
   \item[(i)] $T$ is subnormal,
   \item[(ii)] $|\alpha|^2 + \|E\|^2 \Le 1,$
   \item[(iii)] $T$ is a contraction.
   \end{enumerate}
   \end{cor}
   \section{\label{Sec6}A solution to the Cauchy dual subnormality
problem in the class $\gqb$} We begin by providing a
complete answer to the question of when the Cauchy dual of
an operator of class $\gqb$ is subnormal.
   \begin{thm} \label{main-Ca}
Suppose that $T = \big[\begin{smallmatrix} V & E \\
0 & Q \end{smallmatrix}\big] \in \gqbh$ is left-invertible.
Then $T'$ is subnormal if and only if $\sigma(|Q|, |E|)
\subseteq \big(\mathbb R^2_+ \setminus \mathbb D_+ \big)
\cup \big(\rbb_+ \times \{0\}\big).$
   \end{thm}
   \begin{proof}
Since $T$ is left invertible, we infer from
Proposition~\ref{babuc}(iii) that $\varOmega_1$ is
invertible and
   \begin{align} \label{tipr-0}
\sigma(|Q|,|E|) \subseteq \big\{(s,t)\in \rbb_+^2\colon s^2
+ t^2 \Ge \|\varOmega_1^{-1}\|^{-1}\big\}.
   \end{align}
Therefore, the function $\boldsymbol{\psi}\colon
\sigma(|Q|,|E|) \to \rbb^2$ given by
   \begin{align*}
\boldsymbol{\psi}(s,t) = \Big(\frac{s}{s^2+t^2},
\frac{t}{s^2+t^2}\Big), \quad (s,t) \in \sigma(|Q|,|E|),
   \end{align*}
is well defined and continuous. By
Proposition~\ref{babuc}(iv), $T'\in \gqbh$ and
   \begin{align} \label{tipr-2}
T'= \begin{bmatrix} V  & \tilde E \\
0 & \tilde Q \end{bmatrix},
   \end{align}
where $\tilde E:=E \varOmega_1^{-1}$ and $\tilde
Q:=Q\varOmega_1^{-1}.$ It is easily seen that
   \begin{align*}
\text{$|\tilde Q| = |Q|(|Q|^2+|E|^2)^{-1}$ and $|\tilde E|
= |E|(|Q|^2+|E|^2)^{-1}.$}
   \end{align*}
Using the Stone-von Neumann functional calculus and
Theorem~\ref{spmt}(iii), we obtain
   \begin{align}  \label{tipr-3}
\sigma(|\tilde Q|,|\tilde E|) =
\sigma(\boldsymbol{\psi}(|Q|,|E|)) =
\boldsymbol{\psi}(\sigma(|Q|,|E|)).
   \end{align}
Applying Theorem~\ref{main}(iii) to $T'$ in place of $T$
and using \eqref{tipr-0}, \eqref{tipr-2} and
\eqref{tipr-3}, we complete the proof.
   \end{proof}
We now show that within the class $\gqb$ the Cauchy dual
subnormality problem has an affirmative solution. What is
more surprising is that we can solve it affirmatively even
if complete hyperexpansivity is replaced by expansivity.
For a more detailed discussion of this question, see
Proposition~\ref{bad-suc0} and Example~\ref{not-contr-3}.
The solution is given in Corollary~\ref{cd-im-su} below
which is a direct consequence of Proposition~\ref{kontr-2}
and Theorem~\ref{main-Ca}. Another way of obtaining
Corollary~\ref{cd-im-su} is to apply
Proposition~\ref{babuc}(iv), Corollary~\ref{main-cor} and
the well-known and easy to prove fact that the Cauchy dual
of an expansive operator is a contraction.
   \begin{cor} \label{cd-im-su}
The Cauchy dual of an expansive operator of class $\gqb$ is
a subnormal contraction.
   \end{cor}
   Below we recapture the affirmative solution to the
Cauchy dual subnormality problem for quasi-Brownian
isometries.
   \begin{cor}[\mbox{\cite[Theorem~4.5]{A-C-J-S}}]
The Cauchy dual of a quasi-Brownian isometry is a subnormal
contraction.
   \end{cor}
   \begin{proof}
Let $T\in \ogr{\hh}$ be a quasi-Brownian isometry. If $T$
is an isometry, then $T'=T$ is subnormal. If $T$ is not an
isometry, then by \cite[Proposition~ 5.1]{Maj}, $T$ has the
block matrix form \eqref{brep} with entries satisfying the
conditions \eqref{gqb-1}, \eqref{gqb-2} and \eqref{gqb-3},
$Q$ being an isometry. Since each isometry is quasinormal,
we deduce that $T$ is an operator of class $\gqb$ and
$Q^*Q+E^*E \Ge I$. Combined with Proposition~ \ref{kontr-2}
and Corollary~ \ref{cd-im-su}, this implies that $T'$ is a
subnormal contraction, which completes the proof.
   \end{proof}
Regarding Corollaries~\ref{main-cor} and \ref{cd-im-su}, it
is worth pointing out that there are subnormal operators of
class $\gqb$ that are not contractive, and non-expansive
left-invertible operators of class $\gqb$ whose Cauchy dual
operators are subnormal. This can be deduced from
Theorems~\ref{main}(iii) and \ref{main-Ca} and
Propositions~\ref{kontr-2} and \ref{babuc}(iii) via an
abstract non-explicit procedure given in
Theorem~\ref{bielan-2}. Explicit instances are given in
Example~\ref{not-contr-1} below which will be continued in
Sections~\ref{Sec8} and \ref{Sec9} under different
circumstances.
   \begin{exa} \label{not-contr-1}
Our goal in this example is to show that
   \begin{enumerate}
   \item[$1^{\circ}$] {\em for any $\theta\in (1,\infty),$ there exists a
subnormal operator $T$ of class $\gqb$ such that
$\|T\|=\theta,$}
   \item[$2^{\circ}$] {\em for any $\vartheta\in (0,1),$ there exists
$T = \big[\begin{smallmatrix} V & E \\
0 & Q \end{smallmatrix}\big] \in \gqbh$ such that $T$ is
left-invertible, $T'$ is subnormal and
$\|\varOmega_1^{-1}\|^{-1}=\vartheta$ $($cf.\
\eqref{dudu-1} and \eqref{ale-num}$).$}
   \end{enumerate}
For this purpose, let $\kk$ be an infinite dimensional
complex Hilbert space and $\tau,\eta$ be complex numbers
such that $\eta \neq 0.$ Take a non-unitary isometry $V\in
\ogr{\kk}$ and a quasinormal operator $\tilde Q \in
\ogr{\ob{V}}.$ Define the operators $Q_{\tau}, E_{\eta} \in
\ogr{\kk}$ by
   \begin{align*}
Q_{\tau} = \tau I_{\jd{V^*}} \oplus \tilde Q \quad
\text{and} \quad E_{\eta} = \eta P,
   \end{align*}
where $P\in \ogr{\kk}$ is the orthogonal projection of
$\kk$ onto $\jd{V^*}.$ Then the operator $Q_{\tau}$ is
quasinormal. It is easily seen that
$T_{\tau,\eta} :=  \big[\begin{smallmatrix} V & E_{\eta} \\
0 & Q_{\tau} \end{smallmatrix}\big] \in \gqbhh{\kk}$ (see
Definition~\ref{defq}). The operators $|Q_{\tau}|$ and
$|E_{\eta}|$ can be represented relative to the orthogonal
decomposition $\kk=\jd{V^*} \oplus \ob{V}$ as~ follows:
   \begin{align} \label{dy-du-pu}
|Q_{\tau}|=|\tau| I_{\jd{V^*}} \oplus |\tilde Q|,
\quad |E_{\eta}| = |\eta| I_{\jd{V^*}} \oplus 0.
   \end{align}
Since $\jd{V^*} \neq \{0\},$ we infer from
\eqref{tplus0} and Remark~\ref{many-seu-2} that
   \begin{align} \notag
\sigma(|Q_{\tau}|,|E_{\eta}|) & = \sigma(|\tau|
I_{\jd{V^*}},|\eta| I_{\jd{V^*}}) \cup \sigma(|\tilde Q|,0)
   \\ \label{A-yu-ta}
&= \{(|\tau|,|\eta|)\} \cup \big(\sigma(|\tilde Q|) \times
\{0\}\big).
   \end{align}
According to \eqref{font-1}, \eqref{dy-du-pu} and
\eqref{A-yu-ta}, the following chain of equivalences holds
   \begin{align} \label{sryp-u}
\sigma(|Q_{\tau}|,|E_{\eta}|) = \sigma(|Q_{\tau}|) \times
\sigma(|E_{\eta}|) \iff \sigma(|Q_{\tau}|)=\{|\tau|\} \iff
|Q_{\tau}|=|\tau| I.
   \end{align}
Combined with \eqref{A-yu-ta}, Theorem~\ref{main}(ii)
implies that
   \begin{align} \label{Dr-wr}
\text{\em $T_{\tau,\eta}$ is subnormal if and only if
$(|\tau|,|\eta|) \in \dbbc_+.$}
   \end{align}
Since by \eqref{dy-du-pu},
   \begin{align} \label{Wiki-17-2}
E_{\eta}^*E_{\eta} + Q_{\tau}^*Q_{\tau} = (|\tau|^2 +
|\eta|^2) I_{\jd{V^*}} \oplus |\tilde Q|^2,
   \end{align}
we deduce from Proposition~\ref{babuc}(iii) that
   \begin{align} \label{Wiki-17-3}
\text{\em $T_{\tau,\eta}$ is left-invertible if and only if
$|\tilde Q|$ is invertible.}
   \end{align}
In turn, Theorem~\ref{main-Ca} and \eqref{A-yu-ta} together
yield the following:
   \begin{align} \label{wiki-17}
   \begin{minipage}{65ex}
{\em if $T_{\tau,\eta}$ is left-invertible, then
$T_{\tau,\eta}'$ is subnormal if and only if
$(|\tau|,|\eta|) \in \rbb_+^2 \setminus \dbb_+.$}
   \end{minipage}
   \end{align}
It follows from Proposition~\ref{ogr-1} and \eqref{A-yu-ta}
that (cf.\ \eqref{crak-1})
   \begin{align} \label{wiki-17-a}
\|T_{\tau,\eta}\|& = \max\big\{1, \sqrt{|\tau|^2 +
|\eta|^2}, \|\tilde Q\|\big\}.
   \end{align}
We are now ready to justify $1^{\circ}$ and $2^{\circ}$. If
$\theta \in (1,\infty),$ $(|\tau|,|\eta|) \in \dbbc_+$ and
$\tilde Q$ is chosen so that $\|\tilde Q\|=\theta,$ then in
view of \eqref{Dr-wr} and \eqref{wiki-17-a},
$T_{\tau,\eta}$ is a subnormal operator of class $\gqb$
such that $\|T_{\tau,\eta}\|=\theta,$ which proves
$1^{\circ}$. In turn, if $\vartheta \in (0,1),$
$(|\tau|,|\eta|) \in \rbb_+^2 \setminus \dbb_+$ and $\tilde
Q$ is chosen to be invertible with\footnote{\;Appropriately
translating and rescaling an arbitrary quasinormal operator
does the job.} $\||\tilde Q|^{-1}\|^2 = \vartheta^{-1}$,
then in view of \eqref{Wiki-17-2}, \eqref{Wiki-17-3} and
\eqref{wiki-17}, $T_{\tau,\eta}$ is a left-invertible
operator of class $\gqb$ such that $T_{\tau,\eta}'$ is
subnormal and $\|(E_{\eta}^*E_{\eta} +
Q_{\tau}^*Q_{\tau})^{-1}\|^{-1}=\vartheta,$ which yields
$2^{\circ}$.
   \hfill{$\diamondsuit$}
   \end{exa}
   \section{\label{Sec7}Quasi-Brownian isometries of class $\gqb$}
In this section we provide a few characterizations of
quasi-Brownian isometries of class $\gqb.$ Given an
isometry $V\in \ogr{\hh},$ we say that
$\hh=\hh_1\oplus \hh_2$ is the {\em von Neumann-Wold
decomposition} of $\hh$ for $V$ if
$\hh_1=\bigcap_{n=0}^{\infty} V^n(\hh)$ and
$\hh_2=\bigoplus_{n=0}^{\infty} V^n\jd{V^*}$; recall
that $\hh_1$ reduces $V$ to a unitary operator and
$\hh_2$ reduces $V$ to a unilateral shift of
multiplicity $\dim{\jd{V^*}}$ (see
\cite[Theorem~1.1]{SF70} for more details). It is
clear that
   \begin{align} \label{sam-il-1}
\hh_2=\bigoplus_{n=0}^{\infty} V^n\jd{(V|_{\hh_2})^*}.
   \end{align}
   \begin{thm} \label{q2iso}
Suppose $T =  \big[\begin{smallmatrix} V & E \\
0 & Q \end{smallmatrix}\big] \in \gqbh.$ Then the following
conditions are equivalent{\em :}
   \begin{enumerate}
   \item[(i)] $T$ is a quasi-Brownian isometry,
   \item[(ii)] $T$ is a $2$-isometry,
   \item[(iii)] $(|Q|^2-I)(|Q|^2 + |E|^2 - I)=0,$
   \item[(iv)] $\sigma(|Q|,|E|) \subseteq \mathbb
T_+ \cup \big(\{1\} \times \rbb_+\big),$
   \item[(v)] there exists an orthogonal decomposition
$\hh_2=\hh_{\mathrm{i}} \oplus \hh_{\mathrm{si}}$ $($zero
summands are allowed\/$)$ such that
   \begin{enumerate}
   \item[(a)] $\hh_{\mathrm{i}}$ and $\hh_{\mathrm{si}}$
reduce both $Q$ and $|E|,$
   \item[(b)] $Q|_{\hh_{\mathrm{i}}}$ is an isometry and
$\big(Q|_{\hh_{\mathrm{si}}},|E|\big|_{\hh_{\mathrm{si}}}\big)$
is a spherical isometry.
   \end{enumerate}
   \end{enumerate}
Moreover, if $\hh_{\mathrm{i}}$ and $\hh_{\mathrm{si}}$ are
as in $(${\rm v}$)$ and $\hh_{\mathrm{i}}=\hh_{\mathrm{u}}
\oplus \hh_{\mathrm{s}}$ is the von Neumann-Wold
decomposition of $\hh_{\mathrm{i}}$ for
$Q|_{\hh_{\mathrm{i}}}$, then $\hh_{\mathrm{u}}$ and
$\hh_{\mathrm{s}}$ reduce both $Q$ and $|E|,$
$Q|_{\hh_{\mathrm{u}}}$ is a unitary operator and
$Q|_{\hh_{\mathrm{s}}}$ is a unilateral shift $($of finite
of infinite multiplicity$).$
   \end{thm}
   \begin{proof}
(i)$\Leftrightarrow$(ii) If $T$ is $2$-isometric, then
by \cite[Lemma~1]{R-0}, $T^*T \Ge I.$ This together
with Proposition \ref{babuc}(v) shows that (i) and
(ii) are equivalent.

(ii)$\Leftrightarrow$(iii) This equivalence is a
straightforward consequence of \eqref{gqb-3} and
Proposition~\ref{babuc}(ii).

(iii)$\Leftrightarrow$(iv) Apply Theorem~\ref{spmt}(ii) to
$\psi(x_1,x_2) = (x_1^2 - 1)(x_1^2 + x_2^2 -1)$ and
use~\eqref{font-1}.

(iii)$\Rightarrow$(v) Since $Q$ is quasinormal, $|Q|^2-I$
commutes with $Q$ and so $\hh_{\mathrm{i}}:=\jd{|Q|^2-I}$
reduces $Q$ to an isometry. Set $\hh_{\mathrm{si}}=\hh_2
\ominus \hh_{\mathrm{i}} = \overline{\ob{|Q|^2-I}}$.
Clearly, $\hh_{2} = \hh_{\mathrm{i}} \oplus
\hh_{\mathrm{si}}$ and $\hh_{\mathrm{si}}$ reduces $Q.$
Since $|Q|^2-I$ commutes with $|E|$, we see that
$\hh_{\mathrm{i}},$ and consequently $\hh_{\mathrm{si}},$
reduces $|E|$. Notice that
   \begin{align*}
(|Q|^2 + |E|^2)(|Q|^2-I) = (|Q|^2-I) (|Q|^2 + |E|^2)
\overset{\mathrm{(iii)}}= |Q|^2-I,
   \end{align*}
which implies that $|Q|^2 + |E|^2$ is the identity operator
on $\hh_{\mathrm{si}}$. This shows that
$\big(Q|_{\hh_{\mathrm{si}}},|E|\big|_{\hh_{\mathrm{si}}}\big)$
is a spherical isometry.

(v)$\Rightarrow$(iii) This implication is a matter of
routine verification.

We now prove the ``moreover'' part. Let
$\hh_{\mathrm{i}}=\hh_{\mathrm{u}} \oplus \hh_{\mathrm{s}}$
be the von Neumann-Wold decomposition of $\hh_{\mathrm{i}}$
for $Q|_{\hh_{\mathrm{i}}}.$ Since $\hh_{\mathrm{u}}$ and
$\hh_{\mathrm{s}}$ reduce $Q|_{\hh_{\mathrm{i}}}$ and
$\hh_{\mathrm{i}}$ reduces $Q,$ we deduce that
$\hh_{\mathrm{u}}$ and $\hh_{\mathrm{s}}$ reduce $Q,$ the
operator $Q|_{\hh_{\mathrm{u}}}$ is unitary and the
operator $Q|_{\hh_{\mathrm{s}}}$ is a unilateral shift
$($of finite of infinite multiplicity$).$ Because
$\hh_{\mathrm{u}} = \bigcap_{n=0}^{\infty}
Q^n(\hh_{\mathrm{i}}),$ $|E|(\hh_{\mathrm{i}}) \subseteq
\hh_{\mathrm{i}}$ and $Q$ commutes with $|E|,$ we see that
   \begin{align*}
|E|(\hh_{\mathrm{u}}) \subseteq \bigcap_{n=0}^{\infty}
Q^n|E|(\hh_{\mathrm{i}}) \subseteq \hh_{\mathrm{u}},
   \end{align*}
which implies that $\hh_{\mathrm{u}}$ reduces $|E|.$ Since
$\hh_{\mathrm{si}}$ also reduces $|E|,$ we conclude that
$\hh_{\mathrm{s}}$ reduces $|E|.$ This completes the proof.
   \end{proof}
Below we show that there are operators of class $\gqb$ with
injective $E,$ which are not $2$-isometries (the case when
$E=0$ is obvious due to the fact that quasinormal
$2$-isometries are isometric; see \cite[Theorem 1 in \S
2.6.2]{Furu} and \cite[Theorem~ 3.4]{Ja-St}).
   \begin{cor} \label{qnotquasi}
Suppose $T =  \big[\begin{smallmatrix} V & E \\
0 & Q \end{smallmatrix}\big] \in \gqbh,$ where $E$ is an
isometry. Then the following conditions are equivalent{\em
:}
   \begin{enumerate}
   \item[(i)] $T$ is a $2$-isometry,
   \item[(ii)] $Q=0\oplus U,$
where $U\in \ogr{\hh_2 \ominus \jd{Q}}$ is an isometry.
   \end{enumerate}
   \end{cor}
   \begin{proof}
In view of the equivalence of (ii) and (iii) in Theorem~
\ref{q2iso}, $T$ is a $2$-isometry if and only if
$(Q^*Q)^2=Q^*Q$. Hence, by \cite[Problem~ 127]{Hal}, $T$ is
a $2$-isometry if and only if $Q$ is a partial isometry.
Since $Q$ is quasinormal, we infer from \cite[Problem~
204]{Hal} that $Q$ is a partial isometry if and only if
$Q=0\oplus U,$ where $U\in \ogr{\hh_2 \ominus \jd{Q}}$ is
an isometry.
   \end{proof}
Taking any quasinormal operator $Q$ which is not of the
form as in the condition (ii) of Corollary~\ref{qnotquasi}
(e.g., when $\|Q\|\notin \{0,1\}$), we get an operator of
class $\gqb$ which is not a $2$-isometry.

The key role which plays the Taylor spectrum
$\sigma(|Q|,|E|)$ in the present paper raises the
question of the existence of different orthogonal
decompositions of the underlying Hilbert space $\hh$
relative to which a given operator $T\in \ogr{\hh}$ is
of class $\gqb,$ i.e., $T$ has the block matrix form
\eqref{brep} with $V,$ $E$ and $Q$ satisfying
\eqref{gqb-1}-\eqref{gqb-4}. This question is
discussed in the following example.
   \begin{exa} \label{wid-dmo-1}
Set $Y=\tbb_+ \cup \big(\{1\} \times \rbb_+\big).$ Let
$\varGamma$ be any nonempty compact subset of $Y$ such that
   \begin{align} \label{Wik-1}
\text{$\varGamma \cap \big(\{1\} \times (0,\infty)\big)\neq
\emptyset.$}
   \end{align}
Set $\alpha = \max\{t\in \rbb_+\colon (1,t) \in
\varGamma\}.$ By \eqref{Wik-1}, $\alpha > 0.$ It follows
from Theorem~\ref{bielan-2} that there exists $T =
\big[\begin{smallmatrix} V & E \\
0 & Q \end{smallmatrix}\big] \in \gqbh$ such that
   \begin{align}  \label{kaw-a1}
\sigma(|Q|,|E|) = \varGamma.
   \end{align}
Since $\varGamma \subseteq Y,$ we infer from
Theorem~\ref{q2iso} that $T$ is a quasi-Brownian isometry.
According to \eqref{Wik-1} and Proposition~\ref{kontr-2},
$T$ is not an isometry. Thus using
\cite[Proposition~5.1]{Maj},
we see that  $T =  \big[\begin{smallmatrix} \tilde V & \tilde E \\
0 & \tilde Q \end{smallmatrix}\big] \in
\EuScript{Q}_{\tilde \hh_1,\tilde \hh_2}$ relative to an
orthogonal decomposition $\hh=\tilde \hh_1 \oplus \tilde
\hh_2,$ where $\tilde Q$ is an isometry. Consequently,
   \begin{align} \label{kaw-a2}
\sigma(|\tilde Q|,|\tilde E|) = \sigma(I_{\tilde
\hh_2},|\tilde E|) \overset{\eqref{tplus0}}= \{1\} \times
\sigma(|\tilde E|).
   \end{align}
In view of \eqref{Wik-1}, \eqref{kaw-a1} and
Proposition~\ref{ogr-1} (see also Remark~\ref{Wik-2}), we
have
   \begin{align*}
\|T\| = r(|Q|,|E|) = r(|\tilde Q|,|\tilde E|) = \sqrt{1 +
\alpha^2},
   \end{align*}
where
   \begin{align} \label{sra-1}
\alpha = \max\{t\in \rbb_+\colon & (1,t) \in
\sigma(|Q|,|E|)\}.
   \end{align}
This, together with \eqref{kaw-a2}, implies that
   \begin{align} \label{trab-a}
\alpha=\max \sigma(|\tilde E|) = \||\tilde E|\| = \|\tilde
E\|.
   \end{align}
Since by \eqref{sra-1}, $(1,\alpha)\in \sigma(|Q|,|E|),$ we
infer from \eqref{pro-pr-y} that $\alpha\in \sigma(|E|).$
Consequently,
   \begin{align} \label{zen-rig}
\|E\|=\||E|\|=\max \sigma(|E|) \Ge \alpha
\overset{\eqref{trab-a}} = \|\tilde E\|.
   \end{align}

We now consider two important cases. First, if $\tbb_+
\subseteq \varGamma,$ then by \eqref{kaw-a1} and
\eqref{kaw-a2} we obtain the two block matrix
representations of $T$, namely
   \begin{align*}
\text{$T = \begin{bmatrix} V & E \\
0 & Q \end{bmatrix} \in \gqbh$ relative to $\hh=\hh_1
\oplus \hh_2$}
   \end{align*}
and
   \begin{align*}
\text{$T = \begin{bmatrix} \tilde V & \tilde E \\
0 & \tilde Q \end{bmatrix} \in \EuScript{Q}_{\tilde
\hh_1,\tilde \hh_2}$ relative to $\hh=\tilde \hh_1 \oplus
\tilde \hh_2$,}
   \end{align*}
such that
   \begin{align*}
Z \subseteq \sigma(|Q|,|E|) \quad \text{and} \quad Z \cap
\sigma(|\tilde Q|,|\tilde E|) = \emptyset,
   \end{align*}
where $Z:=\tbb_+ \setminus \{(1,0)\}.$ Second, if $\alpha >
1,$ then using \eqref{zen-rig}, the inclusion $\varGamma
\subseteq Y$ and \eqref{pro-pr-y}, we deduce that $\alpha
=\max \sigma(|E|);$ hence $\alpha = \||E|\| = \|E\|$ which
by \eqref{trab-a} yields $\alpha=\|E\|=\|\tilde E\|.$
   \hfill{$\diamondsuit$}
   \end{exa}
   \section{\label{Sec7.5}Brownian isometries of class $\gqb$}
The aim of this section is to give a deeper insight into
the structure of Brownian isometries of class $\gqb$. We
begin by proving two preparatory lemmata which are of some
independent interest.
   \begin{lem} \label{snia-pr-C}
Suppose $T =  \big[\begin{smallmatrix} V & E \\
0 & Q \end{smallmatrix}\big] \in \gqbh.$ Then the operators
$|Q|$, $|E|$ and $|Q^*|$ commute and the following
conditions are equivalent{\em :}
   \begin{enumerate}
   \item[(i)] $T$ is a Brownian isometry,
   \item[(ii)] $(|Q|^2-I)(|Q|^2 + |E|^2 - I)=0$ and
$(|Q^*|^2-I)(|Q|^2 + |E|^2 - I)^2=0.$
   \end{enumerate}
   \end{lem}
   \begin{proof}
That the operators $|Q|$, $|E|$ and $|Q^*|$ commute can be
deduced from \eqref{gqb-3} and \eqref{gqb-4} via the square
root theorem (cf.\ Proposition~\ref{mem-ber-1}(i)). Hence,
by the equivalence (ii)$\Leftrightarrow$(iii) of Theorem
\ref{q2iso}, it suffices to show that $\triangle_T
\triangle_{T^*} \triangle_T=0$ if and only if
$(|Q^*|^2-I)(|Q|^2 + |E|^2 - I)^2=0.$ It is a routine
matter to verify that \allowdisplaybreaks
   \begin{align*}
\triangle_T \triangle_{T^*} \triangle_T &=
\left[\begin{array}{cc} 0 & 0 \\[1ex] 0 & (\varOmega_1 - I)
(|Q^*|^2-I)(\varOmega_1 - I)
\end{array}\right]
   \\
&= \left[\begin{array}{cc} 0 & 0 \\[1ex] 0 &
(|Q^*|^2-I)(\varOmega_1 - I)^2
\end{array}\right],
   \end{align*}
where $\varOmega_1=|Q|^2 + |E|^2.$ As a consequence, we get
the desired equivalence.
   \end{proof}
   \begin{lem} \label{q2iso-N}
Suppose $T =  \big[\begin{smallmatrix} V & E \\
0 & Q \end{smallmatrix}\big] \in \gqbh$ is a quasi-Brownian
isometry. Let $\hh_2=\hh_{\mathrm{i}} \oplus
\hh_{\mathrm{si}}$ be an orthogonal decomposition of
$\hh_2$ $($zero summands are allowed\/$)$ satisfying the
conditions {\em (a)} and {\em (b)} of Theorem~{\em
\ref{q2iso}} and let $\hh_{\mathrm{i}}=\hh_{\mathrm{u}}
\oplus \hh_{\mathrm{s}}$ be the von Neumann-Wold
decomposition of $\hh_{\mathrm{i}}$ for
$Q|_{\hh_{\mathrm{i}}}.$ Then $T$ is a Brownian isometry if
and only if $|E|\big|_{\hh_{\mathrm{s}}}= 0.$
   \end{lem}
   \begin{proof}
Suppose $T$ is a Brownian isometry. Then by
Lemma~\ref{snia-pr-C}, we have
   \begin{align} \label{wod-a-1}
0 = (I-|Q^*|^2)(|Q|^2 + |E|^2 -
I)^2|_{\hh_{\mathrm{s}}} = P A_{\mathrm{s}}^4 = (P
A_{\mathrm{s}})^4,
   \end{align}
where $A_s:= |E|\big|_{\hh_{\mathrm{s}}}$ and $P\in
\ogr{\hh_{\mathrm{s}}}$ is the orthogonal projection
of $\hh_{\mathrm{s}}$ onto $\jd{Q_{\mathrm{s}}^*}$
with $Q_{\mathrm{s}}:=Q|_{\hh_{\mathrm{s}}}$ (by the
moreover part of Theorem~\ref{q2iso},
$\hh_{\mathrm{s}}$ reduces $Q$ and $|E|$). Because $(P
A_{\mathrm{s}})^*=P A_{\mathrm{s}}$, we infer from
\eqref{wod-a-1} that $P A_{\mathrm{s}}=0.$ As a
consequence, we see that $\ob{A_{\mathrm{s}}}
\subseteq \jd{Q_{\mathrm{s}}^*}^{\perp}$. Since
$Q_{\mathrm{s}}$ commutes with $A_{\mathrm{s}},$ so
does $Q_{\mathrm{s}}^*$ and consequently
$A_{\mathrm{s}}(\jd{Q_{\mathrm{s}}^*}) \subseteq
\jd{Q_{\mathrm{s}}^*}.$ Putting all of this together,
we see that
$A_{\mathrm{s}}(\jd{Q_{\mathrm{s}}^*})=\{0\}.$
Therefore, because $Q$ commutes with $|E|,$ we deduce
that $|E|Q^n \jd{Q_{\mathrm{s}}^*} = \{0\}$ for all
$n\in \zbb_+.$ Since by \eqref{sam-il-1},
$\hh_{\mathrm{s}} = \bigoplus_{n=0}^{\infty}
Q^n\jd{Q_{\mathrm{s}}^*}$, we conclude that
$|E|\big|_{\hh_{\mathrm{s}}}=0.$

To prove the converse implication, assume that
$|E|\big|_{\hh_{\mathrm{s}}}=0.$ It follows from Theorem~
\ref{q2iso} that $\hh_2 =\hh_{\mathrm{u}} \oplus
\hh_{\mathrm{s}} \oplus \hh_{\mathrm{si}}$, the spaces
$\hh_{\mathrm{u}},$ $\hh_{\mathrm{s}}$ and
$\hh_{\mathrm{si}}$ reduce both $Q$ and $|E|,$
$Q|_{\hh_{\mathrm{u}}}$ is unitary, $Q|_{\hh_{\mathrm{s}}}$
is a unilateral shift and $\big(Q|_{\hh_{\mathrm{si}}},|E|
\big|_{\hh_{\mathrm{si}}}\big)$ is a spherical isometry.
Now, straightforward calculations show that the condition
(ii) of Lemma~\ref{snia-pr-C} holds. Hence by this lemma,
$T$ is a Brownian isometry. This completes the proof.
   \end{proof}
   We are now ready to characterize Brownian isometries of
class $\gqb.$
   \begin{thm} \label{snia-pr}
Suppose $T =  \big[\begin{smallmatrix} V & E \\
0 & Q \end{smallmatrix}\big] \in \gqbh.$ Then the following
conditions are equivalent{\em :}
   \begin{enumerate}
   \item[(i)] $T$ is a Brownian isometry,
   \item[(ii)] $(|Q|^2-I)(|Q|^2 + |E|^2 - I)=0$ and
$(|Q^*|^2-I)(|Q|^2 + |E|^2 - I)=0,$
   \item[(iii)] there exists an orthogonal decomposition
$\hh_2=\hh_{\mathrm{u}} \oplus \hh_{\mathrm{s}} \oplus
\hh_{\mathrm{si}}$ $($zero summands are allowed\/$)$ such
that
   \begin{enumerate}
   \item[(a)] $\hh_{\mathrm{u}},$ $\hh_{\mathrm{s}}$ and
$\hh_{\mathrm{si}}$ reduce both $Q$ and $|E|,$
   \item[(b)] $Q|_{\hh_{\mathrm{u}}}$ is a unitary operator and
$Q|_{\hh_{\mathrm{s}}}$ is a unilateral shift $($of finite
or infinite multiplicity$),$
   \item[(c)] $\big(Q|_{\hh_{\mathrm{si}}},|E|
\big|_{\hh_{\mathrm{si}}}\big)$ is a spherical isometry,
   \item[(d)] $|E|\big|_{\hh_{\mathrm{s}}}=0.$
   \end{enumerate}
   \end{enumerate}
   \end{thm}
   \begin{proof}
(i)$\Rightarrow$(iii) Since any Brownian isometry is a
quasi-Brownian isometry, it follows from
Theorem~\ref{q2iso} that there exists an orthogonal
decomposition $\hh_2=\hh_{\mathrm{i}} \oplus
\hh_{\mathrm{si}}$ (zero summands are allowed) satisfying
the conditions (a) and (b) of Theorem~\ref{q2iso}(v). Let
$\hh_{\mathrm{i}}=\hh_{\mathrm{u}} \oplus \hh_{\mathrm{s}}$
be the von Neumann-Wold decomposition of $\hh_{\mathrm{i}}$
for $Q|_{\hh_{\mathrm{i}}}.$ By the moreover part of
Theorem~\ref{q2iso}, the orthogonal decomposition
$\hh_2=\hh_{\mathrm{u}} \oplus \hh_{\mathrm{s}} \oplus
\hh_{\mathrm{si}}$ satisfies the conditions (a), (b) and
(c). Applying Lemma~ \ref{q2iso-N}, we conclude that (d)
holds.

(iii)$\Rightarrow$(ii) This can be shown by straightforward
calculations.

(ii)$\Rightarrow$(i) This implication is a direct
consequence of Lemma~\ref{snia-pr-C}.
   \end{proof}
The following corollary is a consequence of
Theorem~\ref{q2iso}, Lemma~\ref{q2iso-N} and the uniqueness
part of \cite[Theorem~1.1]{SF70} (see also the proof of
Theorem~\ref{snia-pr}).
   \begin{cor} \label{neander}
Suppose $T =  \big[\begin{smallmatrix} V & E \\
0 & Q \end{smallmatrix}\big] \in \gqbh.$ Then the following
conditions are equivalent{\em :}
   \begin{enumerate}
   \item[(i)] $T$ is a quasi-Brownian isometry which is
not a Brownian isometry,
   \item[(ii)] there exists an orthogonal decomposition
$\hh_2=\hh_{\mathrm{u}} \oplus \hh_{\mathrm{s}} \oplus
\hh_{\mathrm{si}}$ $($zero summands are allowed\/$)$ such
that
   \begin{enumerate}
   \item[(a)] $\hh_{\mathrm{u}},$ $\hh_{\mathrm{s}}$ and
$\hh_{\mathrm{si}}$ reduce both $Q$ and $|E|,$
   \item[(b)] $Q|_{\hh_{\mathrm{u}}}$ is a unitary operator and
$Q|_{\hh_{\mathrm{s}}}$ is a unilateral shift $($of finite
or infinite multiplicity$),$
   \item[(c)] $\big(Q|_{\hh_{\mathrm{si}}},|E|
\big|_{\hh_{\mathrm{si}}}\big)$ is a spherical isometry,
   \item[(d)] $|E|\big|_{\hh_{\mathrm{s}}}\neq 0.$
   \end{enumerate}
   \end{enumerate}
   \end{cor}
As shown below, the class of Brownian isometries is
the only subclass of $\gqb$ considered in this paper
which cannot be characterized by the Taylor spectrum
$\sigma(|Q|,|E|)$ of the pair $(|Q|,|E|).$
   \begin{rem} \label{haha-2}
Notice that the condition (ii) of
Theorem~\ref{snia-pr} is equivalent to the conjunction
of the following two inclusions
   \begin{align} \label{dra-1}
   \begin{aligned}
\sigma(|Q|, |E|) & \subseteq \mathbb T_+ \cup (\{1\}
\times \mathbb R_+),
   \\
\sigma(|Q|, |E|, |Q^*|) & \subseteq \big(\mathbb T_+
\times \mathbb R_+ \big) \cup (\mathbb R_+^2 \times
\{1\}),
   \end{aligned}
   \end{align}
where $\sigma(|Q|, |E|, |Q^*|)$ stands for the Taylor
spectrum of $(|Q|, |E|, |Q^*|)$ (recall that the
operators $|Q|$, $|E|$ and $|Q^*|$ commute; see
Lemma~\ref{snia-pr-C}). In view of
Theorem~\ref{q2iso}, it remains to show that the
equation $(|Q^*|^2-I)(|Q|^2 + |E|^2 - I)=0$ is
equivalent to the second inclusion in \eqref{dra-1}.
However, this is immediate from \eqref{font-1} and the
spectral mapping theorem applied to the polynomial $p$
in three variables given by
   \begin{align*}
p(s, t, r)=(r^2-1)(s^2+t^2-1).
   \end{align*}

We conclude this remark by reexamining \cite[Example~
4.4]{A-C-J-S}. Let $V\in \ogr{\hh_1},$ $E\in
\ogr{\hh_2, \hh_1}$ and $Q\in \ogr{\hh_2}$ be
isometric operators such that $Q$ is not unitary and
$V^*E=0.$ As shown in \cite[Example~ 4.4]{A-C-J-S},
the operator $T$ defined by \eqref{brep} is a
quasi-Brownian isometry (obviously of class $\gqb$)
which is not a Brownian isometry. Clearly, the first
inclusion in \eqref{dra-1} holds. Hence by the above
discussion the second one does not hold. The latter
also follows directly from the equality $\sigma(|Q|,
|E|, |Q^*|)=\{1\} \times \{1\} \times \{0,1\}$ which
is a consequence of the projection property of the
Taylor spectrum. Regarding Corollary~\ref{neander},
note that $\hh_2=\hh_{\mathrm{u}} \oplus
\hh_{\mathrm{s}},$ $\hh_{\mathrm{si}}=\{0\},$
$\hh_{\mathrm{s}} \neq \{0\}$ and
$|E|\big|_{\hh_{\mathrm{s}}} \neq 0.$ Summarizing, the
operator $T$ is a quasi-Brownian isometry which is not
a Brownian isometry and $\sigma(|Q|, |E|)=\{(1,1)\}.$
On the other hand, if $\tilde T$ is any non-isometric
Brownian isometry, then it is a Brownian-type operator
of class $\mathcal U$ (see the remark just after
Definition~\ref{defq}), i.e.,
$\tilde T = \big[\begin{smallmatrix} \tilde V & \tilde E \\
0 & \tilde Q \end{smallmatrix}\big] \in
\EuScript{Q}_{\tilde \hh_1,\tilde \hh_2}$ relative to
an orthogonal decomposition $\tilde \hh_1 \oplus
\tilde \hh_2,$ where $\tilde Q$ is a unitary operator.
As a consequence, $\sigma(|\tilde Q|, |\tilde
E|)=\{(1,1)\}.$ This means that Brownian isometries
cannot be characterized by the Taylor spectrum
$\sigma(|Q|, |E|).$
   \hfill{$\diamondsuit$}
   \end{rem}
   \section{\label{Sec8}$m$-isometries and
related operators of class $\gqb$}
   In this section we characterize $m$-contractions,
$m$-isometries and $m$-expansions of class $\gqb$ by using
the Taylor spectrum approach.

Given an integer $m \Ge 1$ and an operator $T\in
\ogr{\hh},$ we write
   \begin{align*} \bscr_m(T) = \sum_{j=0}^m (-1)^j
{m \choose j}{T^*}^jT^j.
   \end{align*}
Recall that an operator $T\in \ogr{\hh}$ is
   \begin{enumerate}
   \item[$\bullet$] {\em $m$-contractive} if $\bscr_m(T) \Ge 0,$
   \item[$\bullet$] {\em $m$-expansive} if $\bscr_m(T) \Le 0,$
   \item[$\bullet$] {\em $m$-isometric} if $T$ is
$m$-contractive and $m$-expansive, that is $\bscr_m(T) =
0,$
   \item[$\bullet$] {\em completely hyperexpansive} if $T$
is $m$-expansive for all $m\Ge 1.$
   \end{enumerate}
The above-mentioned concepts can be attributed to many
authors, such as Agler \cite{Ag} ($m$-contractivity),
Richter \cite{R-0} ($2$-expansivity), Aleman \cite{Alem93}
(complete hyperexpansivity for special operators), Agler
\cite{Ag-0} ($m$-isometricity) and Athavale \cite{Ath}
($m$-expansivity and complete hyperexpansivity). It is
well-known that a $2$-isometry is $m$-isometric for every
integer $m\Ge 2$ (see \cite[Paper I, \S 1]{Ag-St}).
Combined with \cite[Lemma 1(a)]{R-0}, this implies that
each $2$-isometry is completely hyperexpansive. On the
other hand, Agler proved in \cite[Theorem~3.1]{Ag} that an
operator $T\in \ogr{\hh}$ is a subnormal contraction if and
only if it is {\em completely hypercontractive}, i.e., $T$
is $m$-contractive for every positive integer $m.$

   The expression $\bscr_m(T)$ for an operator $T$ of class
$\gqb$ can be described as follows.
   \begin{lem} \label{prob-1}
Suppose that $T = \big[\begin{smallmatrix} V & E \\
0 & Q \end{smallmatrix}\big] \in \gqbh.$ Then
   \begin{align*}
\bscr_m(T) = \begin{bmatrix}  0 & 0 \\[1ex]
0 & \psi_m(|Q|,|E|) \end{bmatrix}, \quad m\in \nbb,
   \end{align*}
where $\psi_m\colon \rbb_+^2 \to \rbb$ are polynomial
functions defined by
   \begin{align} \label{chmur-b1}
\psi_m(s,t) = \Big(1-s^2 - t^2\Big) (1-s^2)^{m-1},
\quad (s,t)\in \rbb_+^2, \, m \in \nbb.
   \end{align}
   \end{lem}
   \begin{proof}
For $m\in \nbb,$ we set $\varLambda_m = \sum_{j=0}^m
(-1)^j {m \choose j}\varOmega_j,$ where $\varOmega_j$
are as in \eqref{dudu-1}. In view of
Proposition~\ref{babuc}(ii), we have
   \begin{align} \label{kra-kra}
\bscr_m(T) = \begin{bmatrix} 0 & 0
   \\
0 & \varLambda_m
\end{bmatrix}, \quad m\in \nbb.
   \end{align}
Let $G$ be the joint spectral measure of $(|Q|,|E|).$
It follows from \eqref{slipy} and \eqref{form-1}~ that
   \begin{align} \label{dyr-dyr}
\varLambda_m = \int_{\rbb_+^2} \tilde\psi_m d G, \quad
m\in \nbb,
   \end{align}
where $\tilde\psi_m\colon \rbb_+^2 \to \rbb$ are
continuous functions defined by
   \begin{align*}
\tilde\psi_m(s,t) = \sum_{j=0}^m (-1)^j {m \choose j}
\varphi_j(s,t), \quad (s,t) \in \rbb_+^2, \, m\in
\nbb.
   \end{align*}
Now we show that $\tilde\psi_m=\psi_m$ for any $m\in
\nbb.$ For this, note that
   \begin{align*}
\tilde\psi_m(1,t) \overset{\eqref{form-1}}= t^2
\sum_{j=1}^m (-1)^j {m \choose j} j = - m t^2
\sum_{j=0}^{m-1} (-1)^j {m-1 \choose j}, \quad t \in
\rbb_+, \, m \in \nbb.
   \end{align*}
Hence, we get
   \begin{align*}
\tilde\psi_m(1,t)=
   \begin{cases}
-t^2 & \text{if } m=1,
   \\[1ex]
0 & \text{if } m\Ge 2,
   \end{cases}
\quad t \in \rbb_+.
   \end{align*}
In turn if $s\neq 1$, we can argue as follows:
   \begin{align*}
\tilde\psi_m(s,t) & \overset{\eqref{mlask-1}}=
\sum_{j=0}^m (-1)^j {m \choose j}
\left(\frac{t^2}{1-s^2} +
\bigg(1-\frac{t^2}{1-s^2}\bigg)s^{2j}\right)
   \\
& \hspace{.75ex}= \bigg(1-\frac{t^2}{1-s^2}\bigg)
\sum_{j=0}^m (-1)^j {m \choose j} s^{2j}
   \\
& \hspace{.75ex} = (1-s^2 - t^2) (1-s^2)^{m-1}, \quad
(s,t) \in (\rbb_+\setminus \{1\}) \times \rbb_+, \,
m\in \nbb.
   \end{align*}
Putting all this together we see that $\tilde\psi_m
=\psi_m$ for all $m\in \nbb.$ Combined with
\eqref{fr-fr-1}, \eqref{kra-kra} and \eqref{dyr-dyr},
this completes the proof.
   \end{proof}
\begin{figure}
\begin{tikzpicture}[scale=.317, transform shape]
\tikzset{vertex/.style = {shape=circle,draw,minimum
size=1em}} \tikzset{edge/.style = {->,> = latex'}}

\node[] (1) at (-6.5, -6.5) {$\scalebox{1.6}{(0, 0)}$};
\node[] (1) at (0.2, -6.5) {$\scalebox{1.6}{(1, 0)}$};

\node[] (1) at (-6.8, -0.1) {$\scalebox{1.6}{(0, 1)}$};

\node[] (1) at (-3.5, -7.5) {$\scalebox{2.2}{\mbox{$m = 1
\phantom{bp}$}}$};

\draw[->, line width=0.05mm] (0,-6) -> (4,-6); \draw[->,
line width=0.05mm] (-6,0) -> (-6,3);

\def\Radius{6}

\filldraw[fill opacity=0.7, line width=0.5mm, fill=gray]
(-6,-6) -- (0,-6) arc (0:90:6cm) -- cycle;
   \end{tikzpicture}
   \begin{tikzpicture}[scale=.317, transform shape]
\tikzset{vertex/.style = {shape=circle,draw,minimum
size=1em}} \tikzset{edge/.style = {->,> = latex'}}

\node[] (1) at (-6.5, -6.5) {$\scalebox{1.6}{(0, 0)}$};
\node[] (1) at (0.2, -6.5) {$\scalebox{1.6}{(1, 0)}$};

\node[] (1) at (-6.8, -0.1) {$\scalebox{1.6}{(0, 1)}$};

\node[] (1) at (-3.5, -7.5) {$\scalebox{2.2}{\mbox{$m\Ge 3$
odd}}$};

\draw[->, line width=0.05mm] (0,-6) -> (4,-6);

\draw[line width=0.5mm] (0,-6) -> (0,3);

\draw[->, line width=0.05mm] (-6,0) -> (-6,3);

\def\Radius{6}

\filldraw[fill opacity=0.7, line width=0.5mm, fill=gray]
(-6,-6) -- (0,-6) arc (0:90:6cm) -- cycle;
   \end{tikzpicture}
   \begin{tikzpicture}[scale=.317, transform shape]
\tikzset{vertex/.style = {shape=circle,draw,minimum
size=1em}} \tikzset{edge/.style = {->,> = latex'}}

\fill [gray] (0, -6) rectangle (4, 3);

\node[] (1) at (-6.5, -6.5) {$\scalebox{1.6}{(0, 0)}$};

\node[] (1) at (0.3, -6.5) {$\scalebox{1.6}{(1, 0)}$};

\node[] (1) at (-6.8, -0.1) {$\scalebox{1.6}{(0, 1)}$};

\node[] (1) at (-3, -7.5) {$\scalebox{2.2}{\mbox{$m\Ge 2$
even}\phantom{d}}$};

\draw[->, line width=0.5mm] (-6,-6) -> (4,-6);

\draw[line width=0.5mm] (0,-6) -> (0,3);

\draw[->, line width=0.05mm] (-6,0) -> (-6,3);

\draw[line width=0.5mm] (0,-6) -> (0,3);

\def\Radius{6}

\filldraw[fill opacity=0.7, line width=0.5mm, fill=gray]
(-6,-6) -- (0,-6) arc (0:90:6cm) -- cycle;
   \end{tikzpicture}
   \caption{Spectral region for $m$-contractivity of
operators of class $\gqb.$} \label{fig8}
   \end{figure}
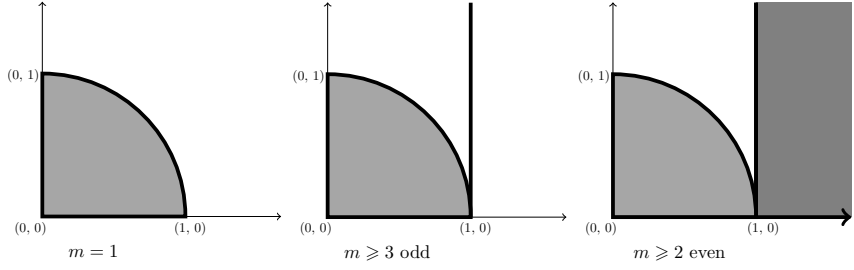
   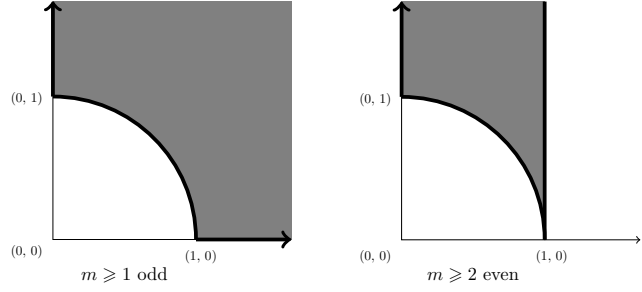
\begin{figure}
   \begin{tikzpicture}[scale=.317, transform shape]
\tikzset{vertex/.style = {shape=circle,draw,minimum
size=1em}} \tikzset{edge/.style = {->,> = latex'}}

\fill [gray] (-6, -6) rectangle (4, 4);

\node[] (1) at (-7.1, -6.5) {$\scalebox{1.6}{(0, 0)}$};
\node[] (1) at (0.2, -6.7) {$\scalebox{1.6}{(1, 0)}$};
\node[] (1) at (-7.1, -0.1) {$\scalebox{1.6}{(0, 1)}$};
\node[] (1) at (-3, -7.5) {$\scalebox{2.2}{\mbox{$m\Ge 1$
odd}}$};

\draw[->, line width=0.5mm] (0,-6) -> (4,-6);

\draw[->, line width=0.5mm] (-6,0) -> (-6,4);

\def\Radius{6}

\filldraw[fill opacity=1, line width=0.1mm, fill=white]
(-6,-6) -- (0,-6) arc (0:90:6cm) -- cycle;

\draw[black, line width=0.5mm] (0,-6) arc (0:90:6);
   \end{tikzpicture}
   \hspace{4ex}
   \begin{tikzpicture}[scale=.317, transform shape]
\tikzset{vertex/.style = {shape=circle,draw,minimum
size=1em}} \tikzset{edge/.style = {->,> = latex'}}

\fill [gray] (-6, -6) rectangle (0, 4);

\node[] (1) at (-7.1, -6.7) {$\scalebox{1.6}{(0, 0)}$};
\node[] (1) at (0.3, -6.7) {$\scalebox{1.6}{(1, 0)}$};
\node[] (1) at (-7.1, -0.1) {$\scalebox{1.6}{(0, 1)}$};
\node[] (1) at (-3, -7.5) {$\scalebox{2.2}{\mbox{$m\Ge 2$
even}}$};

\draw[->, line width=0.05mm] (0,-6) -> (4,-6);

\draw[line width=0.5mm] (0,-6) -> (0,4);

\draw[->, line width=0.5mm] (-6,0) -> (-6,4);

\def\Radius{6}

\filldraw[fill opacity=1, line width=0.1mm, fill=white]
(-6,-6) -- (0,-6) arc (0:90:6cm) -- cycle;

\draw[black, line width=0.5mm] (0,-6) arc (0:90:6);
\end{tikzpicture}
\caption{Spectral region for $m$-expansivity of
operators of class $\gqb.$} \label{fig10}
\end{figure}

   We are now in a position to characterize
$m$-contractivity, $m$-isometricity and
$m$-expansivity of operators of class $\gqb$. The
spectral regions for $m$-contractivi\-ty and
$m$-expansivity of operators of class $\gqb$ are
illustrated in Figures~\ref{fig8} and \ref{fig10} (for
the case $m=1,$ see Proposition~ \ref{kontr-2}).
   \begin{thm} \label{main2}
Assume that $T = \big[\begin{smallmatrix} V & E \\
0 & Q \end{smallmatrix}\big] \in \gqbh$ and $m\Ge 2$
is an integer. Then the following assertions hold{\em
:}
   \begin{enumerate}
   \item[(i)] $T$ is $m$-contractive if and only if
   \begin{align*}
\sigma(|Q|,|E|) \subseteq
   \begin{cases}
\dbbc_+ \cup \big( \{1\} \times \rbb_+\big) & \text{if
} m \text{ is odd,}
   \\[1ex]
\dbbc_+ \cup \big([1,\infty) \times \rbb_+\big) &
\text{if } m \text{ is even,}
   \end{cases}
   \end{align*}
   \item[(ii)] $T$ is $m$-expansive if and only if
   \begin{align*}
\sigma(|Q|,|E|) \subseteq
   \begin{cases}
\rbb_+^2 \setminus \dbb_+ & \text{if } m \text{ is
odd,}
   \\[1ex]
\big(\rbb_+^2 \setminus \dbb_+\big) \cap \big([0,1]
\times \rbb_+\big) & \text{if } m \text{ is even,}
   \end{cases}
   \end{align*}
   \item[(iii)] $T$ is $m$-isometric if and only if
$\sigma(|Q|,|E|) \subseteq \tbb_+ \cup \big(\{1\}
\times \rbb_+\big).$
   \end{enumerate}
   \end{thm}
   \begin{proof} Since the proofs of (i)
and (ii) are similar, we justify only (i). Let $G$ be
the joint spectral measure of the pair $(|Q|,|E|).$
Observe that by \eqref{fr-fr-1} and
Lemma~\ref{prob-1}, $\bscr_m(T)\Ge 0$ if and only if
$\int_{\rbb_+^2} \psi_m d G \Ge 0.$ By
Theorem~\ref{spmt}(i) and Lemma~\ref{main-ad}, the
latter holds if and only if
   \begin{align} \label{coldplay-1}
\sigma(|Q|,|E|) = \supp{G} \subseteq \big\{(s,t)\in
\rbb_+^2\colon \psi_m(s,t) \Ge 0\big\}.
   \end{align}
Using \eqref{chmur-b1}, we verify that
   \begin{align*}
\big\{(s,t)\in \rbb_+^2\colon \psi_m(s,t) \Ge 0\big\}
=
\begin{cases} \dbbc_+ \cup \big( \{1\} \times
\rbb_+\big) & \text{if } m \text{ is odd,}
   \\[1ex]
\dbbc_+ \cup \big([1,\infty) \times \rbb_+\big) &
\text{if } m \text{ is even.}
   \end{cases}
   \end{align*}
Combined with \eqref{coldplay-1}, this yields (i).
Finally, (iii) can be deduced from (i) and (ii). This
completes the proof.
   \end{proof}
   \begin{cor} \label{coldplay-2}
Assume that $T = \big[\begin{smallmatrix} V & E \\
0 & Q \end{smallmatrix}\big] \in \gqbh$ and $m\Ge 2$
is an integer. Then the following assertions hold{\em
:}
   \begin{enumerate}
   \item[(i)] if $m$ is odd $($resp., even$)$, then $T$ is
$m$-contractive if and only if $T$ is $3$-contractive
$($resp., $2$-contractive$)$,
   \item[(ii)] if $m$ is odd $($resp., even$)$, then $T$ is
$m$-expansive if and only if $T$ is expansive
$($resp., $2$-expansive$)$,
   \item[(iii)] $T$ is
$m$-isometric if and only if $T$ is $2$-isometric,
   \item[(iv)] $T$ is completely hypercontractive if and
only if $T$ is contractive,
   \item[(v)] $T$ is completely hyperexpansive if and
only if $T$ is $2$-expansive.
   \end{enumerate}
   \end{cor}
   \begin{proof}
Use Theorem~\ref{main2} and additionally
Proposition~\ref{kontr-2} in the cases (ii), (iv) and (v).
   \end{proof}
   \begin{figure}
   \begin{tikzpicture}[scale=.35, transform shape]
\tikzset{vertex/.style = {shape=circle,draw,minimum
size=1em}} \tikzset{edge/.style = {->,> = latex'}}

\node[] (1) at (-6.5, -6.5) {$\scalebox{1.5}{(0, 0)}$};
\node[] (1) at (0.5, -6.5) {$\scalebox{1.5}{(1, 0)}$};
\node[] (1) at (-6.8, 0.3) {$\scalebox{1.5}{(0, 1)}$};
\node[] (1) at (-3, -8) {$\scalebox{2}{\mbox{quasi-Brownian
isometries}}$}; \node[] (1) at (-3, -8.7)
{$\scalebox{1.3}{\mbox{$||$}}$}; \node[] (1) at (-3, -9.4)
{$\scalebox{2}{\mbox{$2$-isometries}\phantom{p}}$};

\draw[->, line width=0.05mm] (0,-6) -> (3,-6); \draw[line
width=0.5mm] (0,-6) -> (0,3); \draw[->, line width=0.05mm]
(-6,0) -> (-6,3);

\def\Radius{6}

\filldraw[fill opacity=1, line width=0.1mm, fill=white]
(-6,-6) -- (0,-6) arc (0:90:6cm) -- cycle;

\draw[black, line width=0.5mm] (0,-6) arc (0:90:6);
   \end{tikzpicture}
   \begin{tikzpicture}[scale=.35, transform shape]
\tikzset{vertex/.style = {shape=circle,draw,minimum
size=1em}} \tikzset{edge/.style = {->,> = latex'}}

\fill [gray] (-6, -6) rectangle (0, 3);

\node[] (1) at (-6.5, -6.5) {$\scalebox{1.5}{(0, 0)}$};
\node[] (1) at (0.5, -6.5) {$\scalebox{1.5}{(1, 0)}$};
\node[] (1) at (-6.8, 0.3) {$\scalebox{1.5}{(0, 1)}$};
\node[] (1) at (-3, -8)
{$\scalebox{2}{\phantom{b}\mbox{complete
hyperexpansions}}$}; \node[] (1) at (-3, -8.7)
{$\scalebox{1.3}{\mbox{$||$}}$}; \node[] (1) at (-3, -9.4)
{$\scalebox{2}{\mbox{$2$-expansions}}$};

\draw[->, line width=0.05mm] (0,-6) -> (3,-6); \draw[line
width=0.5mm] (0,-6) -> (0,3); \draw[->, line width=0.5mm]
(-6,0) -> (-6,3);

\def\Radius{6}

\filldraw[fill opacity=1, line width=0.1mm, fill=white]
(-6,-6) -- (0,-6) arc (0:90:6cm) -- cycle;

\draw[black, line width=0.5mm] (0,-6) arc (0:90:6);
   \end{tikzpicture}
   \begin{tikzpicture}[scale=.35, transform shape]
\tikzset{vertex/.style = {shape=circle,draw,minimum
size=1em}} \tikzset{edge/.style = {->,> = latex'}}

\fill [gray] (-6, -6) rectangle (3, 3);

\node[] (1) at (-6.5, -6.5) {$\scalebox{1.5}{(0, 0)}$};
\node[] (1) at (0.5, -6.5) {$\scalebox{1.5}{(1, 0)}$};
\node[] (1) at (-6.8, 0.3) {$\scalebox{1.5}{(0, 1)}$};
\node[] (1) at (-3, -8)
{$\scalebox{2}{\mbox{\phantom{bp}}}$}; \node[] (1) at (-3,
-8.7) {$\scalebox{2}{\mbox{\phantom{bp}expansions}}$};
\node[] (1) at (-3, -9.4) {$\scalebox{2}{\phantom{bp}}$};

\draw[->, line width=0.5mm] (0,-6) -> (3,-6); \draw[->,
line width=0.5mm] (-6,0) -> (-6,3);

\def\Radius{6}

\filldraw[fill opacity=1, line width=0.1mm, fill=white]
(-6,-6) -- (0,-6) arc (0:90:6cm) -- cycle;

\draw[black, line width=0.5mm] (0,-6) arc (0:90:6);
   \end{tikzpicture}
   \caption{Spectral regions for some subclasses of the
class $\gqb.$} \label{fignew}
   \end{figure}
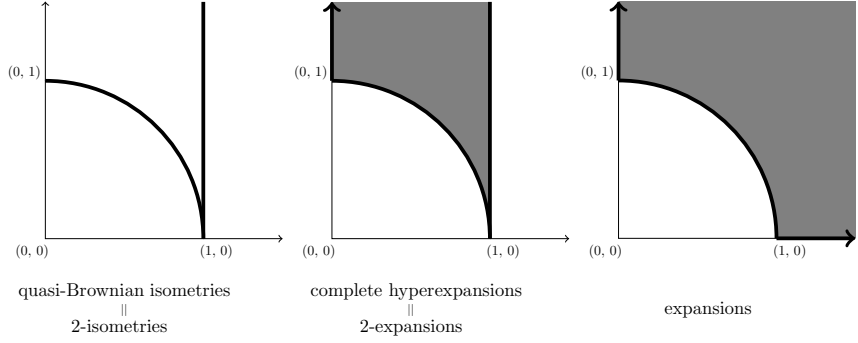
The example below illustrates Theorem~\ref{main2}.
   \begin{exa}[Example~\ref{not-contr-1} continued] \label{not-contr-2}
Let $T_{\tau,\eta}$ be as in
Example~\ref{not-contr-1}. Assume that $m \Ge 2.$
Using \eqref{A-yu-ta}, Proposition~\ref{kontr-2} and
Theorem~\ref{main2}, we get the following:
   \begin{enumerate}
   \item[$1^{\circ}$] if $m$ is odd, then the operator $T_{\tau,\eta}$ is
$m$-contractive if and only if $\tilde Q$ is contractive
and $(|\tau|,|\eta|) \in \dbbc_+ \cup \big( \{1\} \times
\rbb_+\big),$
   \item[$2^{\circ}$] if $m$ is even, then the operator $T_{\tau,\eta}$ is
$m$-contractive if and only if $(|\tau|,|\eta|) \in \dbbc_+
\cup \big([1,\infty) \times \rbb_+\big),$
   \item[$3^{\circ}$] if $m$ is odd, then the operator $T_{\tau,\eta}$ is
$m$-expansive if and only if $T_{\tau,\eta}$ is expansive,
or equivalently if and only if $\tilde Q$ is expansive and
$(|\tau|,|\eta|) \in \rbb_+^2 \setminus \dbb_+,$
   \item[$4^{\circ}$] if $m$ is even, then the operator $T_{\tau,\eta}$ is
$m$-expansive if and only if $\tilde Q$ is an isometry and
$(|\tau|,|\eta|) \in \big(\rbb_+^2 \setminus \dbb_+\big)
\cap \big([0,1] \times \rbb_+\big),$
   \item[$5^{\circ}$] the operator $T_{\tau,\eta}$ is
$m$-isometric if and only if $\tilde Q$ is an isometry and
$(|\tau|,|\eta|) \in \tbb_+ \cup \big(\{1\} \times
\rbb_+\big),$
   \item[$6^{\circ}$] the operator $T_{\tau,\eta}$ is completely
hypercontractive if and only if $\tilde Q$ is a contraction
and $(|\tau|,|\eta|) \in \dbbc_+,$
   \item[$7^{\circ}$] the operator $T_{\tau,\eta}$ is completely
hyperexpansive if and only if $\tilde Q$ is an isometry and
$(|\tau|,|\eta|) \in \big(\rbb_+^2 \setminus \dbb_+\big)
\cap \big([0,1] \times \rbb_+\big).$
   \hfill $\diamondsuit$
   \end{enumerate}
   \end{exa}
   \begin{rem}
Note that in view of \cite[Theorem~2.5]{Gu15}, if
$m\Ge 2$ is even, then any $m$-expansive operator is
$(m-1)$-expansive, while if $m\Ge 3$ is odd, then any
$m$-contractive operator is $(m-1)$-contractive. Using
the assertions $1^{\circ}$--\,$4^{\circ}$ of
Example~\ref{not-contr-2} one can easily show that
none of these implications can be reversed. It is well
known that quasi-Brownian isometries are
$2$-isometric, $2$-isometries are completely
hyperexpansive, complete hyperexpansions are
$2$-expansive and finally $2$-expansions are expansive
(see \cite[Lemma~1]{R-0}). In general, none of these
implications can be reversed. Using
Remark~\ref{haha-2}, Theorem~\ref{q2iso},
Corollary~\ref{coldplay-2} and
Example~\ref{not-contr-2}, one can show that in the
class $\gqb$, these relations take the following form:
   \allowdisplaybreaks
   \begin{align*}
\big\{\text{Brownian isometries in $\gqb$}\big\} &
\varsubsetneq \big\{\text{quasi-Brownian isometries in
$\gqb$}\big\}
   \\
& = \big\{\text{$2$-isometries in $\gqb$}\big\}
   \\
& \varsubsetneq \big\{\text{complete hyperexpansions in
$\gqb$}\big\}
   \\
& = \big\{\text{$2$-expansions in $\gqb$}\big\}
   \\
& \varsubsetneq \big\{\text{expansions in $\gqb$}\big\}.
   \end{align*}
We refer the reader to Figure~\ref{fignew} describing
spectral regions for the above-mentioned subclasses of
the class $\gqb$ (except for Brownian isometries, cf.\
Remark~\ref{haha-2}).
   \hfill{$\diamondsuit$}
   \end{rem}
A recent result due to Badea and Suciu (see
\cite[Theorem~3.4]{B-S}), which states that a
$\triangle_T$-regular $2$-expansive operator $T$ is
completely hyperexpansive if and only if its Cauchy
dual $T'$ is subnormal, solves in the affirmative the
Cauchy dual subnormality problem in the class of
$\triangle_T$-regular $2$-expansions (see
\cite[Theorem~4.5]{A-C-J-S} for an earlier solution of
this problem in the class of $\triangle_T$-regular
$2$-isometries). It is well known and easy to prove
that the relation $T \longleftrightarrow T'$ is a
one-to-one correspondence between expansive operators
and left-invertible contractions. When restricted to
operators of class $\gqb,$ this correspondence becomes
a bijection between expansions and left-invertible
subnormal contractions (see Corollary~\ref{cd-im-su}).
In view of Proposition~\ref{babuc}(v), expansions $T$
of class $\gqb$ are always $\triangle_{T}$-regular.
This suggests that there may exist
$\triangle_{T}$-regular operators outside of the class
of completely hyperexpansive ones for which the Cauchy
dual subnormality problem has an affirmative solution.
This is really the case as shown in
Example~\ref{not-contr-3} below which is based on
Proposition~\ref{bad-suc0}. The proposition itself is
a direct consequence of Propositions~ \ref{kontr-2}
and \ref{babuc}(v), Corollary~\ref{cd-im-su} and
Theorem~\ref{main2}(ii).
   \begin{pro} \label{bad-suc0}
If $T = \big[\begin{smallmatrix} V & E \\ 0 & Q
\end{smallmatrix}\big] \in \gqbh$ is such that
$\sigma(|Q|,|E|) \subseteq \rbb_+^2 \setminus \dbb_+$ and
$\sigma(|Q|,|E|) \cap \big((1,\infty) \times \rbb_+\big)
\neq \emptyset,$ then $T$ is a $\triangle_{T}$-regular
expansion which is not $2$-expansive $($so not completely
hyperexpansive$)$ and whose Cauchy dual $T'$ is a subnormal
contraction.
   \end{pro}
To have a concrete example of an operator satisfying
the assumptions of Proposition~\ref{bad-suc0}, we
revisit Example~\ref{not-contr-1} again.
   \begin{exa}[Example~\ref{not-contr-1} continued] \label{not-contr-3}
Let $T_{\tau,\eta}$ be as in Example~
\ref{not-contr-1}. Suppose that $\tilde Q$ is
expansive and $(|\tau|,|\eta|) \in (1,\infty) \times
(0,\infty).$ Applying \eqref{font-1} and
\eqref{A-yu-ta}, we see that the operator
$T_{\tau,\eta}$ satisfies the hypothesis of
Proposition~\ref{bad-suc0}. As a consequence,
$T_{\tau,\eta}$ is a
$\triangle_{T_{\tau,\eta}}$-regular expansion of class
$\gqb$ such that the Cauchy dual $T_{\tau,\eta}'$ of
$T_{\tau,\eta}$ is a subnormal contraction but
$T_{\tau,\eta}$ itself is not $2$-expansive $($so not
completely hyperexpansive$)$.
   \hfill $\diamondsuit$
   \end{exa}
   \section{\label{Sec9}Linear operator pencils built over the class $\gqb$}
In this section we study linear operator pencils that
are associated with operators of class $\gqb.$ By a
{\em linear operator pencil} (see \cite{G-G-K90}) we
mean a mapping
   \begin{align*}
\varPsi\colon \cbb \ni \lambda \longmapsto A + \lambda
B \in \ogr{\hh},
   \end{align*}
where $A, B\in \ogr{\hh}.$ Given $T =
\big[\begin{smallmatrix} V & E \\ 0 & Q
\end{smallmatrix}\big] \in \gqbh,$ we define the
linear operator pencil $T^{\dag}(\lambda)$ by
   \begin{align*}
T^{\dag}(\lambda) = \begin{bmatrix} V & \lambda E \\
0 & Q \end{bmatrix} = \begin{bmatrix} V & 0 \\
0 & Q \end{bmatrix} +  \lambda \begin{bmatrix} 0 & E \\
0 & 0 \end{bmatrix}, \quad \lambda \in \cbb.
   \end{align*}
Clearly, $T^{\dag}(\lambda)\in \gqbh$ for every $\lambda\in
\cbb.$ Observe that $T^{\dag}(\lambda)$ can be regarded as
the perturbation of the
quasinormal operator  $\big[\begin{smallmatrix} V & 0 \\
0 & Q \end{smallmatrix}\big]$ by the nilpotent operator
$\lambda \big[\begin{smallmatrix} 0 & E \\
0 & 0 \end{smallmatrix}\big]$. It is worth pointing
out that the operators
$\big[\begin{smallmatrix} V & 0 \\
0 & Q \end{smallmatrix}\big]$ and $\big[\begin{smallmatrix} 0 & E \\
0 & 0 \end{smallmatrix}\big]$ do not commute in
general (they commute if and only if $VE=EQ.$) Note
that by Corollary~\ref{gyeongju-c},
$T^{\dag}(\lambda)$ is subnormal if and only if
$T^{\dag}(|\lambda|)$ is subnormal. This justifies why
we concentrate on describing the set $\subE{T}$ given
by
   \begin{align*}
\subE{T}=\big\{\alpha \in \rbb_+\colon T^{\dag}(\alpha)
\text{ is subnormal}\big\}.
   \end{align*}
If, moreover, $E\neq 0$ and $\sigma_{\sharp}(|Q|,|E|)
\subseteq [0,1] \times (0,\infty),$ then we define
$\beta^{\dag}(T)\in \rbb_+$ by (cf.\ \eqref{luft-1})
   \begin{align} \label{Zen-claps}
\beta^{\dag}(T) = \inf_{(s,t)\in \sigma_{\sharp}(|Q|,|E|)}
\sqrt{\frac{1-s^2}{t^2}}.
   \end{align}
Using $\beta^{\dag}(T)$ we can describe the set
$\subE{T}$ explicitly.
   \begin{thm} \label{subex}
Suppose that $T = \big[\begin{smallmatrix} V & E \\ 0 & Q
\end{smallmatrix}\big] \in \gqbh$ and  $E\neq 0.$ Then the
following assertions hold{\em :}
   \begin{enumerate}
   \item[(i)] $0 \in \subE{T},$
   \item[(ii)] $\subE{T} \setminus \{0\} \neq \emptyset$
if and only if $\sigma_{\sharp}(|Q|,|E|) \subseteq [0,1]
\times (0,\infty)$ and $\beta^{\dag}(T) > 0,$
   \item[(iii)] if $\subE{T} \setminus \{0\}  \neq \emptyset,$ then
   \begin{align} \label{zach-1}
\subE{T}=\big[0,\beta^{\dag}(T)\big].
   \end{align}
   \end{enumerate}
Moreover, if $\sigma_{\sharp}(|Q|,|E|) \subseteq [0,1]
\times (0,\infty),$ then
   \begin{align*}
\subE{T}=\{0\} \iff \beta^{\dag}(T)=0.
   \end{align*}
   \end{thm}
   \begin{proof}
First observe that $T^{\dag}(0)$ is a quasinormal operator
and thus by \cite[Proposition~II.1.7]{Co}, $T^{\dag}(0)$ is
subnormal, which yields (i). In view of (i) and
Theorems~\ref{spmt}(iii) and \ref{main}(ii), we have
   \begin{align} \label{okul-Z2}
\subE{T}= \{0\} \cup \big\{\alpha\in (0,\infty)\colon
s^2+\alpha^2 t^2 \Le 1, \; \forall (s,t) \in
\sigma_{\sharp}(|Q|,|E|)\big\}.
   \end{align}
It is now a routine matter to show that for any
$\alpha\in\subE{T},$ $[0,\alpha]\subseteq \subE{T}.$

(ii) \& (iii) Suppose that $\alpha \in \subE{T} \setminus
\{0\}.$ Then by \eqref{okul-Z2}, we have
   \begin{align*}
\alpha^2 \Le \frac{1-s^2}{t^2}, \quad (s,t) \in
\sigma_{\sharp}(|Q|,|E|).
   \end{align*}
As a consequence, $\sigma_{\sharp}(|Q|,|E|) \subseteq [0,1]
\times (0,\infty)$ and $0 < \sup \subE{T} \Le
\beta^{\dag}(T).$ Clearly, by \eqref{okul-Z2},
$\beta^{\dag}(T)\in\subE{T} \setminus \{0\}.$ Hence, in
view of the discussion in the previous paragraph,
\eqref{zach-1} holds.

In turn, if $\sigma_{\sharp}(|Q|,|E|) \subseteq [0,1]
\times (0,\infty)$ and $\beta^{\dag}(T) > 0,$ then as above
we verify that $\beta^{\dag}(T) \in \subE{T} \setminus
\{0\}.$

The ``moreover'' part follows from (i) and (ii). This
completes the proof.
   \end{proof}
There is another possibility of associating a linear
operator pencil with an operator of class $\gqb.$
Namely, given $T = \big[\begin{smallmatrix} V & E \\ 0
& Q
\end{smallmatrix}\big] \in \gqbh,$ we define the pencil
$T_{\dag}(\cdot)$~by
   \begin{align*}
T_{\dag}(\lambda) = \begin{bmatrix} V & E \\
0 & \lambda  Q \end{bmatrix} = \begin{bmatrix} V & E \\
0 & 0 \end{bmatrix} +  \lambda \begin{bmatrix} 0 & 0 \\
0 & Q \end{bmatrix}, \quad \lambda \in \cbb,
   \end{align*}
and the corresponding set $\subQ{T}$ by
   \begin{align*}
\subQ{T}=\big\{\alpha \in \rbb_+\colon T_{\dag}(\alpha)
\text{ is subnormal}\big\}.
   \end{align*}
As before, by Corollary~\ref{gyeongju-c},
$T_{\dag}(\lambda)$ is subnormal if and only if
$T_{\dag}(|\lambda|)$ is subnormal, so we can concentrate
on describing the set $\subQ{T}$. Obviously,
$T_{\dag}(\lambda)\in \gqbh$ for every $\lambda\in \cbb.$
The operator $T_{\dag}(\lambda)$ can be
regarded as the  perturbation of $\big[\begin{smallmatrix} V & E \\
0 & 0 \end{smallmatrix}\big]$ by the quasinormal
operator $\lambda \big[\begin{smallmatrix} 0 & 0 \\
0 & Q \end{smallmatrix}\big]$ (in view of
Theorem~\ref{pen-2}(i),
$\big[\begin{smallmatrix} V & E \\
0 & 0 \end{smallmatrix}\big]$ is subnormal provided
$\subQ{T} \neq \emptyset$). Note also that
   \begin{align} \label{ajjaj-1}
\big[\begin{smallmatrix} V & E  \\
0 & 0  \end{smallmatrix}\big] \textit{ and }
\big[\begin{smallmatrix} 0 & 0 \\
0 & Q \end{smallmatrix}\big] \textit{ commute } \iff EQ=0
\iff |Q||E|=0.
   \end{align}
Indeed, the former equivalence is a consequence of
straightforward calculations while the latter follows from
the identities:
   \begin{align*}
(|Q||E|)^2 \overset{(*)}=Q^*Q E^*E
\overset{\eqref{gqb-3}}= (EQ)^*EQ,
   \end{align*}
where $(*)$ is a consequence of
Proposition~\ref{mem-ber-1}(i).

We are now in a position to describe the set $\subQ{T}.$
   \begin{thm} \label{pen-2}
Suppose that $T = \big[\begin{smallmatrix} V & E \\ 0 & Q
\end{smallmatrix}\big] \in \gqbh.$ Set
   \begin{align*}
\sigma_{\flat}(|Q|,|E|) = \sigma(|Q|,|E|) \cap
\big((0,\infty)\times (0,\infty)\big).
   \end{align*}
Then the following assertions hold{\em :}
   \begin{enumerate}
   \item[(i)] $\subQ{T} \neq \emptyset$  if and only
if $\|E\| \Le 1,$
   \item[(ii)] if $\|E\|\Le 1$ and $\sigma_{\flat}(|Q|,|E|) = \emptyset,$
then $\subQ{T}=\rbb_+,$
   \item[(iii)] if $\|E\|\Le 1$ and $\sigma_{\flat}(|Q|,|E|)
\neq \emptyset,$ then $\subQ{T}=[0,\beta_{\dag}(T)],$
where\/\footnote{\;It follows from $\|E\|\Le 1,$
\eqref{font-1} and \eqref{pro-pr-y} that
$\sigma_{\flat}(|Q|,|E|) \subseteq (0,\infty) \times
[0,1],$ which implies that $\beta_{\dag}(T)$ is well
defined.}
   \begin{align} \label{Zen-claps-2}
\beta_{\dag}(T) := \inf_{(s,t)\in \sigma_{\flat}(|Q|,|E|)}
\sqrt{\frac{1-t^2}{s^2}}.
   \end{align}
   \end{enumerate}
   \end{thm}
   \begin{proof}
It follows from Theorems~\ref{spmt}(iii) and
\ref{main}(ii) that
   \begin{align} \label{num-1}
\subQ{T} = \big\{\alpha \in \rbb_+\colon \alpha^2 s^2 + t^2
\Le 1, \; \forall (s,t) \in \sigma_{\sharp}(|Q|,|E|)\big\}.
   \end{align}
Recall that the set $\sigma_{\sharp}(|Q|,|E|)$ may be
empty (see \eqref{luft-1}). It is easily seen that
   \begin{align} \label{trzy-3}
\text{$[0,\alpha] \subseteq \subQ{T}$ whenever
$\alpha\in\subQ{T}.$}
   \end{align}

(i) In view of \eqref{luft-1} and \eqref{num-1}, there is
no loss of generality in assuming that
$\sigma_{\sharp}(|Q|,|E|) \neq \emptyset.$ Suppose that
$\subQ{T} \neq \emptyset.$ Then, by \eqref{pro-pr-y} and
\eqref{num-1}, $\sigma(|E|)\setminus \{0\} \subseteq
[0,1],$ hence by \eqref{font-1}, $\|E\|\Le 1.$ Conversely,
if $\|E\| \Le 1,$ then by \eqref{font-1}, \eqref{pro-pr-y}
and \eqref{num-1}, $0\in \subQ{T}.$

(ii) Since $\|E\|\Le 1,$ we infer from \eqref{font-1}
that $\sigma(|E|) \subseteq [0,1].$ Hence, if
$(s,t)\in \sigma_{\sharp}(|Q|,|E|) \cap (\{0\} \times
\rbb_+),$ then by using \eqref{pro-pr-y} we see that
$\alpha^2 s^2 + t^2 \Le 1$ for all $\alpha \in
\rbb_+.$ By assumption
   \begin{align*}
\sigma_{\sharp}(|Q|,|E|) \cap \big((0,\infty) \times
\rbb_+\big) = \sigma_{\flat}(|Q|,|E|) = \emptyset,
   \end{align*}
and consequently, by \eqref{num-1}, $\subQ{T}
=\rbb_+.$

(iii) Suppose that $\|E\|\Le 1$ and
$\sigma_{\flat}(|Q|,|E|) \neq \emptyset.$ If $\alpha \in
\subQ{T},$ then by \eqref{num-1}, we have
   \begin{align*}
\alpha^2 \Le \frac{1-t^2}{s^2}, \quad (s,t) \in
\sigma_{\flat}(|Q|,|E|),
   \end{align*}
which implies that $\sup \subQ{T} \Le
\beta_{\dag}(T).$ Now we prove the opposite
inequality. As in (ii), we see that if $(s,t)\in
\sigma_{\sharp}(|Q|,|E|) \cap (\{0\} \times \rbb_+),$
then $\alpha^2 s^2 + t^2 \Le 1$ for all $\alpha \in
\rbb_+.$ In turn, if
   \begin{align*}
(s,t)\in \sigma_{\sharp}(|Q|,|E|) \cap ((0,\infty) \times
\rbb_+)=\sigma_{\flat}(|Q|,|E|),
   \end{align*}
then the inequality $\alpha^2 s^2 + t^2 \Le 1$ holds
for $\alpha=\beta_{\dag}(T).$ Therefore, by
\eqref{num-1}, $\beta_{\dag}(T)\in \subQ{T}.$ Combined
with \eqref{trzy-3}, this implies that
$\subQ{T}=[0,\beta_{\dag}(T)],$ which completes the
proof.
   \end{proof}
   \begin{rem}
Concerning Theorem~\ref{pen-2}, it is worth mentioning that
according to the assertions \eqref{stis-zero} and
\eqref{ajjaj-1} we have
   \begin{align*}
\sigma_{\flat}(|Q|,|E|) = \emptyset \iff |Q||E|=0 \iff
\big[\begin{smallmatrix} V & E \\
0 & 0 \end{smallmatrix}\big] \textit{ and }
\big[\begin{smallmatrix} 0 & 0 \\
0 & Q \end{smallmatrix}\big] \textit{ commute.}
   \end{align*}
In other words, the set $\sigma_{\flat}(|Q|,|E|)$ is empty
if and only if the perturbation $T_{\dag}(\lambda)$ of
$\big[\begin{smallmatrix} V & E \\
0 & 0 \end{smallmatrix}\big]$ commutes with the perturbing
operator $\lambda\big[\begin{smallmatrix} 0 & 0 \\
0 & Q \end{smallmatrix}\big]$ for some $\lambda\in
\cbb\setminus \{0\}.$
   \hfill $\diamondsuit$
   \end{rem}
We now show that for an arbitrary $b\in \rbb_+,$ there
exists  $T =  \big[\begin{smallmatrix} V & E \\
0 & Q \end{smallmatrix}\big] \in \gqbh$ such that
$\subE{T}=[0,b].$ Similarly, for a given $b\in \rbb_+
\cup \{\infty\},$ we can find
$T =  \big[\begin{smallmatrix} V & E \\
0 & Q \end{smallmatrix}\big] \in \gqbh$ such that
$\subQ{T}=[0,b] \cap \rbb_+.$
   \begin{exa}[Example~\ref{not-contr-1} continued]
Let $T_{\tau,\eta}$ be as in Example~\ref{not-contr-1}. We
begin by showing that for every $b \in \rbb_+,$ there exist
$\tau \in \cbb$ and $\eta\in \cbb\setminus \{0\}$ such that
$\subE{T_{\tau,\eta}}=[0,b].$ Indeed, it follows from
\eqref{A-yu-ta} that
$\sigma_{\sharp}(|Q_{\tau}|,|E_{\eta}|) =
\{(|\tau|,|\eta|)\}.$ Assume additionally that $|\tau| \Le
1.$ Combined with \eqref{Zen-claps}, this gives
   \begin{align} \label{num-obiad}
\beta^{\dag}(T_{\tau,\eta}) =
\sqrt{\frac{1-|\tau|^2}{|\eta|^2}}.
   \end{align}
First, suppose that $b=0.$ Then by considering the
case $|\tau|=1$ we infer from the moreover part of
Theorem~\ref{subex} that $\subE{T_{\tau,\eta}}=[0,b].$
Let now $b>0.$ Then by taking into account the case
$|\tau|< 1$ we deduce from \eqref{num-obiad} and
Theorem~\ref{subex} that $\subE{T_{\tau,\eta}} =
\big[0,\beta^{\dag}(T_{\tau,\eta})\big].$ This
together with \eqref{num-obiad} shows that there
exists $\eta \in \cbb\setminus \{0\}$ such that
$\beta^{\dag}(T_{\tau,\eta})=b.$

Similarly, using Theorem~\ref{pen-2}, one can show that for
every $b \in \rbb_+\cup \{\infty\},$ there exist parameters
$\tau$ and $\eta$ such that $\subQ{T_{\tau,\eta}}=[0,b]
\cap \rbb_+.$ We leave the details to the reader.
   \hfill $\diamondsuit$
   \end{exa}
   We conclude this paper by commenting the contents of
this section. In view of Theorem~\ref{main2}, the technique
of using the Taylor spectrum developed here can also be
applied to describe the sets of the form
   \begin{align*}
\big\{\alpha \in \rbb_+\colon T^{\dag}(\alpha) \text{ is in
$ \mathscr C$}\big\} \text{ and } \big\{\alpha \in
\rbb_+\colon T_{\dag}(\alpha) \text{ is in $ \mathscr
C$}\big\},
   \end{align*}
where $\mathscr C$ is one of the classes of operators
appearing in Section \ref{Sec8} including $m$-contractions,
$m$-expansions, etc. As the number of cases to be
considered is large (in particular depends on the parity of
$m$) and each of them requires separate treatment, we
decided not to include details in this paper.

\vspace{1ex}

   {\bf Acknowledgement.} A part of this paper was
written while the first and third authors visited
Jagiellonian University in Summer of 2019, and the
second and fourth authors visited Kyungpook National
University in Autumn of 2019. They wish to thank the
faculties and the administrations of these units for
their warm hospitality.
   
   \end{document}